\documentclass[10pt]{article}
\usepackage{a4wide}
\usepackage{amsmath}
\usepackage{amssymb}
\usepackage{latexsym}
\usepackage[pdftex]{graphicx}
\usepackage{color}
\usepackage[mathscr]{euscript}
\usepackage{url}
\usepackage{float}
\usepackage{subfigure}
\usepackage{morefloats}
\usepackage{wasysym}
\usepackage{mathtools}
\usepackage{enumitem}
\usepackage{booktabs}

\newenvironment{proof}{\vspace{1ex}\noindent{\bf Proof:}}{\hspace*{\fill}$\blacksquare$\vspace{1ex}}
\newenvironment{proofof}[1]{\vspace{1ex}\noindent{\bf Proof of #1:}}{\hspace*{\fill}$\blacksquare$\vspace{1ex}}
\def\noproof{\hspace*{\fill}$\blacksquare$}

\newtheorem{theorem}{Theorem}[section]
\newtheorem{lemma} [theorem] {Lemma}
\newtheorem{corollary} [theorem] {Corollary}

\newtheorem{proposition} [theorem] {Proposition}
\newtheorem{definition} [theorem] {Definition}

\newtheorem{conjecture} [theorem] {Conjecture}

\newcommand{\Acal}[0]{\ensuremath{{\mathcal A}}}
\newcommand{\Bcal}[0]{\ensuremath{{\mathcal B}}}
\newcommand{\Ccal}[0]{\ensuremath{{\mathcal C}}}

\newcommand{\Fcal}[0]{\ensuremath{{\mathcal F}}}
\newcommand{\Gcal}[0]{\ensuremath{{\mathcal G}}}
\newcommand{\Hcal}[0]{\ensuremath{{\mathcal H}}}

\newcommand{\Lcal}[0]{\ensuremath{{\mathcal L}}}

\newcommand{\Ncal}[0]{\ensuremath{{\mathcal N}}}

\newcommand{\Pcal}[0]{\ensuremath{{\mathcal P}}}

\newcommand{\Ucal}[0]{\ensuremath{{\mathcal U}}}

\newcommand{\eN}[0]{\ensuremath{ \mathbb N}}

\newcommand{\Tscr}[0]{\ensuremath{{\mathscr T}}}

\newcommand{\Pee}[0]{\ensuremath{{\mathbb P}}}
\newcommand{\Ee}[0]{\ensuremath{{\mathbb E}}}

\newcommand{\isd}[0]{\hspace{.2ex} \raisebox{-.1ex}{$=$} \hspace{-1.5ex}
\raisebox{1ex}{{$\scriptstyle d$}} \hspace{.8ex} }

 \newcommand{\eps}{\varepsilon}

\DeclareMathOperator{\dist}{dist}

\DeclareMathOperator{\clo}{cl}

\DeclareMathOperator{\Po}{Po}

\DeclareMathOperator{\Var}{Var}
\DeclareMathOperator{\aut}{aut}

\newcommand{\EF}[0]{Ehrenfeucht-Fra\"{\i}ss\'e}
\newcommand{\equivRMSO}[0]{\ensuremath{\equiv^{\mathrm{rMSO}}}}
\newcommand{\equivMSO}[0]{\ensuremath{\equiv^{\mathrm{MSO}}}}

\newcommand{\equivFO}[0]{\ensuremath{\equiv^{\mathrm{FO}}}}
\DeclareMathOperator{\EHR}{EHR}
\DeclareMathOperator{\Biggg}{Big}
\DeclareMathOperator{\Frag}{Frag}
\newcommand{\EHRRMSO}[0]{\ensuremath{\EHR^{\mathrm{rMSO}}}}
\newcommand{\EHRMSO}[0]{\ensuremath{\EHR^{\mathrm{MSO}}}}

\newcommand{\EHRFO}[0]{\ensuremath{\EHR^{\mathrm{FO}}}}
\newcommand{\distTV}[0]{\ensuremath{\dist_{\mathrm{TV}}}}
\DeclareMathOperator{\qd}{qd}
\DeclareMathOperator{\Long}{Long}
\DeclareMathOperator{\Bush}{Bush}
\DeclareMathOperator{\Short}{Short}

\newcommand{\upe}{\mathrm{e}}

\newcommand{\FO}{\textsc{FO}}
\newcommand{\MSO}{\textsc{MSO}}

\DeclareMathOperator{\fw}{fw}

\title{Logical limit laws for minor-closed classes of graphs}

\author{Peter Heinig\thanks{Mathematisches Seminar, Universit\"at Hamburg, Hamburg, Germany. E-mail: \texttt{peter.heinig@math.uni-hamburg.de}.
The author gratefully acknowledges the support of TUM Graduate School's Thematic Graduate Center TopMath
at Technische Universit\"at M\"unchen.}
\and
Tobias M\"uller\thanks{Groningen University, Groningen, the Netherlands. E-mail: \texttt{tobias.muller@rug.nl}.
The author's research was supported in part by an NWO VIDI grant.}
\and
Marc Noy\thanks{Barcelona Graduate School of Mathematics and Universitat Polit\`ecnica de Catalunya, Barcelona, Spain. E-mail: \texttt{marc.noy@upc.edu}.
The author's research was supported by the Humboldt Foundation, by TUM through a Von Neumann visiting professorship,
and by grants MTM2011-13320, MTM2014-54745-P and MDM-2014-0445.}
\and
Anusch Taraz\thanks{Hamburg University of Technology, Hamburg, Germany. E-mail: \texttt{taraz@tuhh.de}.
The author was supported in part by DFG grant TA 319/2-2.}
}

\begin{document}

\maketitle

\begin{abstract}
Let $\Gcal$ be an addable, minor-closed class of graphs. 
We prove that the zero-one law holds in monadic second-order  logic (MSO) for the random graph drawn
uniformly at random from all {\em connected} graphs in $\Gcal$ on $n$ vertices, and the convergence law in MSO
holds if we draw uniformly at random from all graphs in $\Gcal$ on $n$ vertices.
We also prove analogues of these results for the class of graphs embeddable on a fixed surface, provided we restrict
attention to first order logic (FO).  Moreover, the limiting probability that a given FO sentence is satisfied is independent
of the surface $S$.
We also prove that the closure of the set of limiting probabilities is always the finite union of at least two disjoint intervals, and
that it is the same for FO and MSO. For the classes of forests and planar graphs we are able to determine
the closure of the set of limiting probabilities precisely. For planar graphs it consists of exactly 108 intervals, each of the same length
$\approx 5.39 \cdot 10^{-6}$.
Finally, we analyse examples of non-addable classes where the behaviour is quite different. For instance, the zero-one law does not hold for the random caterpillar on $n$ vertices, even in FO.

\vspace{1ex}

{\em Keywords:} random graphs, graph minors, logical limit laws.
\end{abstract}

\section{Introduction}


We say that a sequence $(G_n)_n$ of random graphs obeys the \textit{zero-one law} with respect to some logical language $L$ if for every sentence $\varphi$ in $L$ the probability that a graph $G_n$ satisfies $\varphi$ tends either to 0 or 1, as $n$ goes to infinity.
We say that $(G_n)_n$ obeys the \textit{convergence law} with respect to $L$, if for every $\varphi$ in $L$, the probability
that $G_n$ satisfies $\varphi$ tends to a limit (not necessarily zero or one) as $n$ tends to infinity.

The prime example of a logical language is the {\em first order language of graphs} (\FO).
Formulas in this language are constructed using variables $x, y, \dots$ ranging over
the vertices of a graph, the usual quantifiers $\forall, \exists$, the usual logical connectives $\neg, \vee, \wedge$, etc., parentheses and the binary relations $=, \sim$,
where $x\sim y $ denotes that $x$ and $y$ are adjacent.
For notational convenience we will also allow the use of commas and semicolons in the formulas in this paper.
In $\FO$ one can for instance write ``$G$ is triangle-free" as 
$\neg\exists x,y,z: (x\sim y)\wedge(x\sim z)\wedge(y\sim z)$.

The classical example of a zero-one law is a result due to Glebskii et al.~\cite{GlebskiiEtAl69} and
independently to Fagin~\cite{Fagin76}, stating that the zero-one law holds when $G_n$ is chosen uniformly at random among all
$2^{n \choose 2}$ labelled graphs on $n$ vertices and the language is \FO.
The (non-)existence of $\FO$-zero-one and convergence laws has been investigated more generally
in the $G(n,p)$ binomial model, where there are $n$ labelled vertices and edges are drawn independently with probability~$p$
(the case $p=1/2$ of course corresponds to the uniform distribution on all labelled graphs on
$n$ vertices).
Here, the $\FO$-zero-one law holds for all constant $p$  and in many other cases.
In particular, a remarkable result of Shelah and Spencer says that if $p=n^{-\alpha}$ with $0\le\alpha\le1$ fixed, then
the $\FO$-zero-one law holds if and only if $\alpha$ is an irrational number \cite{SS}.
What is more, when $\alpha$ is rational then the {\em convergence law} in fact fails. That is, for rational $\alpha$
Shelah and Spencer were able to construct an ingenious $\FO$-sentence such that the probability that $G(n,n^{-\alpha})$ satisfies it oscillates
between zero and one.
The book \cite{logicofrandomgraphs} by Spencer contains a detailed account of the story of logical limit laws for
the $G(n,p)$ binomial model.

There are also several results available on zero-one laws for random graphs that satisfy some global condition, such as being regular, having bounded degree or 
being $H$-free for some fixed graph $H$. One of the earliest instances of such results deals with $K_{t+1}$-free graphs: Kolaitis, Pr\"{o}mel and Rothschild \cite{KPR} proved an FO-zero-one law for the graph chosen uniformly at random from all $K_{t+1}$-free graphs on $n$ vertices.
For random $d$-regular graphs, with $d$ fixed, there cannot be a zero-one law. For
instance, the probability of containing a triangle is expressible in $\FO$ and tends to a constant different from 0 and 1.
(This is an immediate consequence of Theorem 9.5 on page 237 of~\cite{purplebook}).
Using the configuration model, Lynch \cite{lynch} proved the \textit{$\FO$-convergence law} for 
graphs with a given degree sequence (subject to some conditions on this degree sequence), which in particular covers $d$-regular graphs with 
$d$ fixed.
For dense $d$-regular graphs, that is when $d = \Theta(n)$, a zero-one law was proved
by Haber and Krivelevich \cite{HK}, who were also able to obtain an analogue of the
striking result of Shelah-Spencer mentioned above for random regular graphs.

In a different direction, McColm \cite{McColm2002}
considers 
trees sampled
uniformly at random from all (labelled) trees on $n$ vertices.
He shows that a zero-one law holds in \textit{monadic second order} (MSO) language of graphs, which is FO enriched with quantification over sets of
vertices. That is, we now have additional variables $X, Y, \dots$ ranging over sets of vertices and
an additional binary relation $\in$, that allows us to ask whether $x \in X$, for $x$ a vertex-variable and $X$ a set-variable.
This results in a stronger language which is able to express properties such as connectedness or $k$-colorability
for any fixed $k$.
In \MSO\ we can for instance express ``$G$ is connected" as \newline
$\forall X : (\forall x: x\in X) \vee \neg(\exists x: x \in X) \vee (\exists x,y: (x\in X)\wedge\neg(y\in X)\wedge (x\sim y))$.

The proof in \cite{McColm2002} is based on two facts: for each $k>0$ there exists a rooted tree $T_k$ such that, if two trees $A$ and $B$ 
both have $T_k$ as a rooted subtree, then $A$ and $B$ agree on all sentences of quantifier depth at most $k$ 
(defined in Section~\ref{sec:logicalpreliminaries} below); and the fact that with high probability a random tree contains~$T_k$ as a rooted subtree. We show that this approach can be adapted to a much more general setting, as we explain next.

A class $\Gcal$ of graphs is \textit{minor-closed} if every minor of a graph in $\Gcal$ is also in $\Gcal$.
Every minor-closed class is characterized by the set of its excluded minors, which is finite by the celebrated Robertson-Seymour theorem.
Notable examples of minor-closed graphs are forests, planar graphs and graphs embeddable on a fixed surface.
We say that $\Gcal$ is addable if it holds that {\bf 1)} $G \in \Gcal$ if and only if all its components are in $\Gcal$, and {\bf 2)}
if $G'$ is obtained from $G\in\Gcal$ by adding an arbitrary edge between two separate components of $G$, then $G' \in \Gcal$ also.
Planar graphs constitute an addable class of graphs, but graphs embeddable on a surface other than the sphere may not.
(A 5-clique is for instance embeddable on the torus, but the vertex-disjoint union of two 5-cliques is not -- see \cite[Theorem 4.4.2]{graphs-surfaces}). 
Other examples of addable classes are outerplanar graphs, series-parallel graphs, graphs with bounded tree-width, and graphs with given 3-connected components~\cite{3-conn}.
Relying heavily on results of McDiarmid~\cite{McDiarmid09}, we are able to prove the MSO-zero-one law for the random graph
chosen uniformly at random from all \emph{connected} graphs on $n$ vertices in an addable minor-closed class.
If $\Gcal$ is a class of graphs, then $\Ccal \subseteq \Gcal$ will denote the set of
{\em connected} graphs from $\Gcal$; the notation $\Gcal_n$ will denote the set of all graphs in $\Gcal$ with vertex set $[n] := \{1,\dots,n\}$, and
$\Ccal_n$ is defined analogously. If $A$ is a finite set then $X \in_u A$ denotes that $X$ is chosen uniformly at random from $A$.

\begin{theorem}\label{thm:MSO01add}
Let $\Gcal$ be an addable, minor-closed class of graphs and let $C_n \in_u \Ccal_n$
be the random {\em connected} graph from $\Gcal$.
Then $C_n$ obeys the $\MSO$-zero-one law.
\end{theorem}

It is worth mentioning that the MSO-zero-one law does not hold in $G(n,1/2)$ -- see \cite{KaufShe}.
For the ``general" random graph from an addable class $\Gcal$, i.e.~$G_n\in_u\Gcal_n$,  there cannot be a zero-one law, even in FO logic.
The reason is that there are sentences expressible in \FO, such as the existence of an isolated vertex, that have a limiting probability 
strictly between 0 and 1.  
(This follows immediately from Theorem~\ref{thm:addfrag} below, which is due to~\cite{McDiarmid09}.) 
We are however able to prove the \emph{convergence} law in this case.

\begin{theorem}\label{thm:MSOconvadd}
Let $\Gcal$ be an addable, minor-closed class of graphs and let $G_n \in_u \Gcal_n$.
Then $G_n$ obeys the $\MSO$-convergence law.
\end{theorem}

The proof is based on the fact that with high probability there is a ``giant" component of size $n-O(1)$,
and uses the extraordinarily precise description of the limiting distribution of the {\em fragment}, the part of the graph that remains after we
remove the largest component~\cite{McDiarmid09}.

Moving away from addable classes of graphs, let $S$ be a fixed surface and consider the class of graphs that can be embedded in $S$. In this case we prove a zero-one law in FO for connected graphs. The proof is based on recent results on random graphs embeddable on a surface \cite{McDiarmid08,chapuyfusygimenezmoharnoy} and an application of Gaifman's locality theorem (see Theorem~\ref{thm:gaifman} below).
The notation $G \models \varphi$ means ``$G$ satisfies $\varphi$".

\begin{theorem}\label{thm:surface01}
Fix a surface $S$, let $\Gcal$ be the class of all graphs embeddable on $S$ and let
$C_n \in_u \Ccal_n$ be the random {\em connected} graph from $\Gcal$.
Then $C_n$ obeys the $\FO$-zero-one law.
Moreover, the values of the limiting probabilities
$\displaystyle \lim_{n\to\infty}\Pee( C_n\models\varphi )$
do not depend on the surface $S$.
\end{theorem}

We remark that an analogous result was proved for random \textit{maps} (connected graphs with a given embedding) on a fixed surface \cite{01maps}.
Again for arbitrary graphs we prove a convergence law in FO.
Moreover, we show that the limiting probability of an FO sentence does not depend on the surface and is the same as for planar graphs.

\begin{theorem}\label{thm:surfaceconv}
Fix a surface $S$, let $\Gcal$ be the class of all graphs embeddable on $S$ and let
$G_n \in_u \Gcal_n$.
Then $G_n$ obeys the $\FO$-convergence law.
Moreover, the values of the limiting probabilities
$\displaystyle \lim_{n\to\infty}\Pee( G_n\models\varphi )$
do not depend on the surface $S$.
\end{theorem}

\noindent
We conjecture that both Theorems~\ref{thm:surface01} and~\ref{thm:surfaceconv} extend to MSO logic.
See the last section of this paper for a more detailed discussion.

Having proved Theorem~\ref{thm:MSOconvadd}, a natural question
is which numbers $p \in [0,1]$ are limiting probabilities of some $\MSO$ sentence.
The proof of Theorem~\ref{thm:MSOconvadd} provides an expression for the limiting
probabilities in terms of the so-called Boltzmann-Poisson distribution (defined in Section~\ref{sec:preliminariesonminorclosedclasses}), but
at present we are not able  to deduce a complete description of the set of limiting probabilities from it.
We are however able to derive some information on the structure of this set.
Two easy observations are that, since there are only countably many sentences $\varphi \in \MSO$
(since a sentence is a finite string of symbols, taken from a countable alphabet), the set
$\left\{\lim_{n\to\infty} \Pee( G_n \models\varphi) \,:\, \varphi \in \MSO \right\}$ must obviously also be
countable; and since $\lim_{n\to\infty} \Pee( G_n \models\neg\varphi) = 1 - \lim_{n\to\infty} \Pee( G_n \models\varphi)$,
it is symmetric with respect to $1/2$.

Certainly not every number $p\in [0,1]$ is a limiting probability of an $\MSO$ sentence, as there are only countably
many such limiting probabilities.
A natural question
is whether the set of limiting probabilities is at least dense in $[0,1]$.
As it turns out this is never the case (for $\Gcal$ an addable, minor-closed class):

\begin{proposition}\label{prop:gap}
Let $\Gcal$ be an addable, minor-closed class of graphs and let $G_n \in_u \Gcal_n$.
For every $\varphi\in\MSO$ we have either
$\lim_{n\to\infty}\Pee( G_n \models\varphi) \leq 1-\upe^{-1/2} \approx 0.3935$ or
$\lim_{n\to\infty}\Pee( G_n \models\varphi) \geq \upe^{-1/2} \approx 0.6065$.
\end{proposition}

\noindent
Next, one might ask for the topology of the set of limiting probabilities.
Could they for instance form some strange, fractal-like set?
(See the last section of this paper for an example of a model of random graphs, where such things do indeed happen.)
The next theorem shows that the set of limiting probabilities is relatively well-behaved, and
also that the limits of $\FO$-sentences are dense in the set of limits of $\MSO$-sentences. We denote by $\clo(A)$ the topological closure of $A$ in $\mathbb{R}$.

\begin{theorem}\label{thm:finintintro}
Let $\Gcal$ be an addable, minor-closed class of graphs and let $G_n \in_u \Gcal_n$.  Then
\[
\clo{\Big(} {\Big\{} \lim_{n\to\infty}\Pee( G_n\models\varphi) :
\varphi \in \MSO {\Big\} } {\Big)}
=
\clo{\Big(} {\Big\{} \lim_{n\to\infty}\Pee( G_n\models\varphi) :
\varphi \in \FO {\Big\} } {\Big)}
\]\noindent
and this set is a finite union of closed intervals.
\end{theorem}

Let us remark that, in view of Proposition~\ref{prop:gap}, the closure always consists of at least two disjoint intervals.
In Section~\ref{sec:conc} we conjecture that there are at least four disjoint intervals for every addable, minor-closed class.
Let us also observe that it follows from this last theorem and Theorem~\ref{thm:surfaceconv} that the closure of the
$\FO$-limiting probabilities for the random graph embeddable on $S$, for any surface $S$, coincides
with the closure of the $\MSO$-limiting probabilities for the class of planar graphs.

Perhaps surprisingly, the closure of the set of $\MSO$-limiting probabilities can actually be determined exactly for two
important examples.
The first one is the class of forests.

\begin{theorem}\label{thm:forestclosure}
If $\Gcal$ is the class of forests and $G_n \in_u \Gcal_n$ then
\[  \begin{array}{c}
\displaystyle
\clo{\Big(} {\Big\{} \lim_{n\to\infty}\Pee( G_n\models\varphi) :
\varphi \in \MSO {\Big\} } {\Big)}  \\[4mm]
\displaystyle
= {\Big[}0, 1 - (1+\upe^{-1})\upe^{-1/2}{\Big]}
\cup
{\Big[}\upe^{-3/2}, 1-\upe^{-1/2}{\Big]}
\cup
{\Big[}\upe^{-1/2}, 1-\upe^{-3/2}{\Big]}
\cup
{\Big[}(1+\upe^{-1})\upe^{-1/2}, 1{\Big]} \\
\displaystyle
\approx {\Big[}0, .170{\Big]}
\cup
{\Big[} .223, .393 {\Big]}
\cup
{\Big[}.606, .776{\Big]}
\cup
{\Big[}.830, 1{\Big]} .
\end{array} \]
\end{theorem}

\noindent
The second class for which we can
determine the closure of the set of limiting probabilities of $\MSO$ sentences is the class of planar graphs.
To describe the results, we need the
exponential generating function corresponding to a class of graphs $\Gcal$, which is defined as
$G(z) = \sum_{n=0}^\infty |\Gcal_n|\frac{z^n}{n!}$.

\begin{theorem}\label{thm:planarclosureprime}
If $\Gcal$ is the class of planar graphs and $G_n\in_u\Gcal_n$ then
the set
\[ \clo{\Big(} {\Big\{} \lim_{n\to\infty}\Pee( G_n\models\varphi) :
\varphi \in \MSO {\Big\} } {\Big)} \]
\noindent
is the union of 108 disjoint intervals that all have exactly the same length, which is
approximately $5.39\cdot 10^{-6}$.
The endpoints of these intervals are given explicitly
in Theorem~\ref{thm:planarclosure} in terms of
$\rho$ and $G(\rho)$, where $G(z)$ is the exponential generating
function for the set of planar graphs, and $\rho$ is its radius of convergence
(see details in Section~\ref{sec:planarlimits}).
\end{theorem}

Until now we have dealt with addable minor-closed classes and graphs embeddable on a fixed surface, which
in view of Theorems~\ref{thm:surface01} and~\ref{thm:surfaceconv} (see also the results we list in Section~\ref{sec:preliminariesonminorclosedclasses})
behave rather similarly to planar graphs.
It is thus natural to ask to which extent our results can be expected to carry over to the non-addable case.
 Random graphs from several non-addable classes have recently been investigated in~\cite{arXiv:1303.3836v1} and
 the results in that paper demonstrate that they can display behaviour very different from the addable case.
In Section~\ref{sec:nonaddable}, we analyse three examples of non-addable graph classes from the logical limit laws point of view;
and the results there are in stark contrast with the results on addable graph classes.

For $t \in \eN$ a fixed integer, the collection $\Gcal$ of all graphs whose every component has no more than $t$ vertices
is a minor-closed class that is not addable.
Of course now $\Ccal_n$, the set of connected graphs from $\Gcal$ on $n$ vertices, is empty for $n > t$.
So it does not make sense to consider the random connected graph $C_n \in_u \Ccal_n$.
For the ``general" random graph $G_n \in_u \Gcal_n$ we see that, contrary to the addable case,
every $\MSO$ sentence has a limiting probability that is either zero or one.

\begin{theorem}\label{thm:boundedcompsize}
Let $t \in\eN$ be fixed, let $\Gcal$ be the class of all graphs whose components have
at most $t$ vertices, and let $G_n \in_u \Gcal_n$.
Then $G_n$ obeys the $\MSO$-zero-one law.
\end{theorem}

Another simple example of a non-addable minor-closed class of graphs $\Gcal$ is formed by forests of paths (every component is a path).
Note that now we can speak of $C_n \in_u \Ccal_n$, the random path on $n$ vertices, but it is still a rather 
uninteresting object as
the only randomness is in the labels of the vertices.
In Section~\ref{sec:nolimpath} we give an $\MSO$-sentence whose probability of holding for the random path
$C_n$ is zero for even $n$ and one for odd $n$, disproving even the {\em $\MSO$-convergence} law for the random
path. On the other hand, we are able to prove the $\MSO$-convergence law for the ``general" random graph from $\Gcal$.
And, finally it turns out that now, contrary to the addable case, the limiting probabilities are in fact dense in $[0,1]$.

\begin{theorem}\label{thm:paths}
Let $C_n$ be the random path on $n$ vertices, and let $G_n$ be the
random forest of paths on $n$ vertices.
Then the following hold:
\begin{enumerate}
\item\label{itm:paths1} The $\MSO$-convergence law fails for $C_n$;
\item\label{itm:paths2} The $\MSO$-convergence law holds for $G_n$;
\item\label{itm:paths3} $\displaystyle \clo{\Big(} {\Big\{} \lim_{n\to\infty}\Pee( G_n\models\varphi) :
\varphi \in \MSO {\Big\} } {\Big)} = [0,1]$.
\end{enumerate}
\end{theorem}

\noindent
We remark that the $\FO$-zero-one law holds for the ``random path'' $C_n$. This follows immediately from Theorem 2.1.3 in~\cite{logicofrandomgraphs}.

Of course graphs with bounded component size and forests of paths are relatively simple
classes of graphs.
We turn attention to a non-addable minor-closed class of graphs that is slightly more challenging 
to analyze.
A {\em caterpillar} consists of a finite path with zero or more vertices of degree one attached to it.
A forest of caterpillars is a graph all of whose components are caterpillars.
In Section~\ref{sec:catcounter} we construct an $\FO$ sentence such that the probability that the random caterpillar satisfies it
converges, but not to zero or one. We are again able to show that the convergence law holds for the
``general" random graph from the class (the random forest of caterpillars), but we are only able to
do this if we restrict attention to $\FO$ logic -- for what appears to be mainly technical reasons.
And finally, also here the limiting probabilities (of $\FO$ sentences) turn out to be dense in $[0,1]$.

\begin{theorem}\label{thm:caterpillars}
Let $C_n$ be the random caterpillar on $n$ vertices, and let $G_n$ be the
random forest of caterpillars on $n$ vertices.
Then the following hold:
\begin{enumerate}
\item\label{itm:cat1} The $\FO$-zero-one law fails for $C_n$;
\item\label{itm:cat2} The $\FO$-convergence law holds for $G_n$;
\item\label{itm:cat3} $\displaystyle \clo{\Big(} {\Big\{} \lim_{n\to\infty}\Pee( G_n\models\varphi) :
\varphi \in \FO {\Big\} } {\Big)} = [0,1]$.
\end{enumerate}
\end{theorem}

\noindent
Very recently (in the period since the submission of the
present paper) it was shown that the $\FO$-convergence law holds for the ``random caterpillar'' $C_n$  by R.~Lambers
in his MSc thesis~\cite{LambersMSc} that was written under the supervision of the second author.

In the next section we give the notation and results from the literature that we need for the proofs of our main results.
In Section~\ref{sec:lawproofs} we give the proofs of Theorem~\ref{thm:MSO01add} through~\ref{thm:surfaceconv}.
In Section~\ref{sec:limitingprobabilities} we prove Proposition~\ref{prop:gap} and
Theorem~\ref{thm:finintintro} through~\ref{thm:planarclosureprime}.
The proofs of our results on non-addable classes can be found in Section~\ref{sec:nonaddable}.
Finally, in Section~\ref{sec:conc} we give some additional thoughts and open questions that arise from our work.

\section{Notation and preliminaries}

All graphs in this paper will be simple and loopless. Throughout the paper, we write $[n]:=\{1,\dotsc,n\}$.
If $G$ is a graph then we denote its vertex set by $V(G)$ and its edge set by $E(G)$.
Their cardinalities are denoted by $v(G) :=\lvert V(G)\rvert$ and  $e(G) := \lvert E(G)\rvert$.
For $u \in V(G)$, we denote by $B_G(u,r)$, called \emph{$r$-neighbourhood of $u$},
the subgraph of $G$ induced by the set of all vertices of graph
distance at most $r$ from $u$.
When the graph is clear from the context we
often simply write $B(u,r)$. Occasionally we write $x\sim y$ to denote that
$x$ is adjacent to $y$, in some graph clear from the context.

If $G, H$ are graphs, then $G \cup H$ denotes the {\em vertex-disjoint union}.
That is, we ensure that $V(G) \cap V(H) = \emptyset$ (by swapping $H$ for an
isomorphic copy if needed) and then we simply take
$V(G\cup H) = V(G)\cup V(H), E(G\cup H) = E(G) \cup E(H)$.
If $n$ is an integer and $G$ a graph then $nG$ denotes the
vertex-disjoint union of $n$ copies of $G$.

An {\em unlabelled graph} is, formally, an isomorphism class of graphs.
In this paper we deal with both labelled and unlabelled graphs.
We will occasionally be a bit sloppy with the distinction between the two
(for instance, by  taking the vertex-disjoint union of a labelled and an unlabelled
graph) but no confusion will arise. A class of graphs $\Gcal$
always means a collection of graphs closed under isomorphism. We denote by
$\Ucal\Gcal$ the collection of unlabelled graphs corresponding to~$\Gcal$.

Following McDiarmid~\cite{McDiarmid09}, we denote by $\Biggg(G)$
the largest component of $G$. In case of ties we take the lexicographically first
among the components of the largest order (i.e., we look at the labels of the vertices
and take the component in which the smallest label occurs).
The `fragment' $\Frag(G)$ of $G$ is what remains after we remove $\Biggg(G)$.


Recall that a random variable is {\em discrete}
if it takes values in some countable set $\Omega$.
For discrete random variables taking values in the same set $\Omega$,
the {\em total variation distance} is defined as:
\[ \distTV( X, Y ) =  \max_{A\subseteq\Omega} |\Pee( X \in A)-\Pee(Y\in A)|. \]
\noindent
Alternatively, by some straightforward manipulations
(see~\cite[Proposition 4.2]{LevinPeresWilmerBoek}),
we can write $\distTV( X, Y ) = \frac12\sum_{a \in \Omega} | \Pee( X = a ) - \Pee( Y = a ) |$.
If $X, X_1, X_2, \dots$ are discrete random variables,
 we say that {\em $X_n$ tends to $X$ in total variation} (notation: $X_n \to_{\text{TV}} X$) if
$\lim_{n\to\infty} \distTV( X_n, X ) = 0$.

Throughout this paper, $\Po(\mu)$ denotes the Poisson distribution with parameter $\mu$.
We make use of the following incarnation of the Chernoff bounds.
A proof can for instance be found in Chapter~1 of~\cite{PenroseBoek}.

\begin{lemma}\label{lem:chernoff}
Let $Z \isd \Po(\mu)$. Then the following bounds hold:
\begin{enumerate}
 \item For all $k \geq \mu$ we have $\displaystyle \Pee( Z \geq k ) \leq \upe^{-\mu H(k/\mu)}$, and
\item For all $k \leq \mu$ we have $\displaystyle \Pee( Z \leq k ) \leq \upe^{-\mu H(k/\mu)}$,
\end{enumerate}
where $H(x) := x\ln x - x + 1$.\noproof
\end{lemma}

In Section~\ref{sec:limitingprobabilities}, we make use of
a general result on the set of all
sums of subsequences of a given summable sequence of nonnegative numbers.
The following observation goes back a hundred years, to the work of Kakeya \cite{kakeya}.

\begin{lemma}[\cite{kakeya}]\label{lem:subsums}
Let $p_1, p_2, \dots$ be a summable sequence of nonnegative numbers.
If  $p_i \leq \sum_{j>i} p_j$ for every $i \in \eN$, then
\[ \left\{ \sum_{i \in A} p_i : A \subseteq \eN\right\} = \left[0, \sum_{i=1}^\infty p_i\right]. \]
\end{lemma}

From this last lemma it is straightforward to
derive the following observation -- see for instance~\cite[Equation (3) and Proposition 6]{Nitecki}.

\begin{corollary}\label{cor:subsums}
Let $p_1, p_2, \dots$ be a summable sequence of nonnegative numbers,
and suppose there is an $i_0 \in\eN$ such that $p_i \leq \sum_{j>i} p_j$
for all $i  > i_0$. Then the set of sums of subsequences of $(p_n)_n$ is a union
of $2^{i_0}$ translates of the interval $\left[0,\sum_{i>i_0}p_i\right]$, namely:

\begin{equation}\label{eq:explicitintervalsingeneral}
\left\{ \sum_{i \in A} p_i : A \subseteq \eN\right\} =
\bigcup_{x_1,\dots, x_{i_0} \in \{0,1\}}
\left[ x_1p_1+\dots+x_{i_0}p_{i_0}, x_1p_1+\dots+x_{i_0}p_{i_0} + \sum_{i>i_0}p_i\right].
\end{equation}
\end{corollary}

\noindent
We remark that the intervals in the LHS of~\eqref{eq:explicitintervalsingeneral} are not necessarily disjoint, but they 
are disjoint if $p_i > \sum_{j>i} p_j$ for all $i\leq i_0$.

\subsection{Logical preliminaries}\label{sec:logicalpreliminaries}

A variable $x$ in a logical formula is called {\em bound} if it has a quantifier.
Otherwise, it is called {\em free}. A {\em sentence} is a formula without free variables.

The {\em quantifier depth} $\qd(\varphi)$ of an MSO formula $\varphi$ is, informally speaking,
the longest chain of `nestings' of quantifiers. More formally, it is defined
inductively using the axioms
{\bf 1)} $\qd(\neg\varphi) = \qd(\varphi)$, {\bf 2)}
$\qd(\varphi\vee\psi) =
\qd(\varphi\wedge\psi) = \qd(\varphi\Rightarrow\psi) = \max(\qd(\varphi),\qd(\psi))$, {\bf 3)}
$\qd( \exists x : \varphi ) = \qd(\forall x : \varphi ) = \qd( \exists X : \varphi ) = \qd( \forall X : \varphi ) = 1 + \qd(\varphi)$,
{\bf 4)} $\qd( x=y ) = \qd( x\sim y) = \qd( x \in X ) = 0$.

For two graphs $G$ and $H$, the notation $G \equivMSO_k H$ denotes that
every $\varphi \in \MSO$ with $\qd(\varphi) \leq k$ is either
satisfied by both $G$ and $H$ or false in both.
We define $G \equivFO_k H$ similarly.
It is immediate from the definition that $\equivMSO_k, \equivFO_k$
are both equivalence relations on the set of all graphs.
What is more, for every $k$ there are only
finitely many equivalence classes (for a proof see, e.g.,
\cite[Proposition~3.1.3]{EbbinghausFlum}):

\begin{lemma}\label{lem:fixednoeq}
For every $k \in \eN$, the relation $\equivMSO_k$ has
finitely many equivalence classes. The same holds for $\equivFO_k$.
\end{lemma}

The {\em $\MSO$-\EF-game} $\EHRMSO_k( G, H )$ is a two-player game played on two graphs
$G, H$ for $k$ rounds.
The game is played as follows. There are two players, Spoiler and Duplicator.
In each round $1\leq i \leq k$, Spoiler is first to move,  selects one of the
two graphs $G$ or $H$ (and in particular is allowed to switch the graph at any new round),
and does either a vertex-move or a set-move.
That is, Spoiler either selects a single vertex or a subset of the vertices.
If Spoiler did a vertex-move then Duplicator now has to select a vertex from
the graph that Spoiler did not play on, and otherwise Duplicator has to select
a subset of the vertices from the graph Spoiler did not play on.
After $k$ rounds the game is finished.
To decide on the winner, we first need to introduce some additional notation.
Let $I \subseteq \{1, \dots, k\}$ be those rounds in which a vertex-move occurred.
For $i \in I$, let $x_i \in V(G), y_i \in V(H)$ be the vertices that were selected
in round $i$. For $i \not\in I$, let $X_i \subseteq V(G), Y_i \subseteq V(H)$
be the subsets of vertices that were selected in round $i$.
Duplicator has won if the following three conditions are met:
\begin{itemize}
\item[{\bf 1)}] $x_i=x_j$ if and only if $y_i = y_j$, for all $i,j \in I$, and;
\item[{\bf 2)}] $x_ix_j \in E(G)$ if and only if $y_iy_j \in E(H)$, for all $i,j \in I$, and;
\item[{\bf 3)}] $x_i \in X_j$ if and only if $y_i \in Y_j$, for all $i \in I, j \not\in I$.
\end{itemize}%
Otherwise Spoiler has won.
We say that $\EHRMSO_k( G, H )$ {\em is a win for Duplicator} if there
exists a winning strategy for Duplicator. (I.e., no matter how Spoiler plays,
Duplicator can always respond so as to win in the end.)
The $\FO$-\EF-game $\EHRFO_k(G, H)$ is defined just like $\EHRMSO_k(G,H)$,
except that set-moves do not exist in that game.
(So in the $\FO$-game, Duplicator wins if and only if {\bf 1)} and {\bf 2)} are met at the end of the game.)

The following lemma shows the relation between these games and logic.
A proof can for instance be found in~\cite[Theorems~2.2.8 and 3.1.1]{EbbinghausFlum}.

\begin{lemma}\label{lem:EFgame}
$G \equivMSO_k H$ if and only if $\EHRMSO_k( G, H )$ is a win for
Duplicator. \\
Similarly, $G\equivFO_k H$ if and only if $\EHRFO_k(G,H)$ is a win for
Duplicator.
\end{lemma}

The \EF-game is a convenient tool for proving statements about logical (in-) equivalence.
It can be used for instance to prove that the statement `$G$ is connected' cannot
be expressed as an $\FO$-sentence, and that similarly `$G$ has a Hamilton cycle' cannot
be expressed by an $\MSO$-sentence.
(See for instance~\cite[Theorem~2.4.1]{logicofrandomgraphs}.)

The following two standard facts about $\equivMSO_k$ are essential tools in our arguments.
They can for instance be found in~\cite[Theorems 2.2 and 2.3]{ComptonII}, where they are proved in a greater level of generality.

\begin{lemma}\label{lem:disjointunion}
If $H_1\equivMSO_k G_1$ and $H_2\equivMSO_k G_2$ then
$H_1 \cup H_2 \equivMSO_k G_1 \cup G_2$. The same conclusion holds w.r.t.~$\equivFO_k$.
\end{lemma}

\begin{lemma}\label{lem:lotsofcopies}
For every $k \in \eN$ there is an $a = a(k)$ such that the following holds.
For every graph $G$ and every $b \geq a$ we have
$aG \equivMSO_k bG$.
\end{lemma}

Let us observe that the statement $\dist(x,y) \leq r$,
where $\dist$ denotes the graph distance,
can easily be written as a first order formula whose only free variables are $x,y$.
If $\varphi$ is a first order formula then we denote by $\varphi^{B(x,r)}$ the formula in
which all bound variables are `relativised to $B(x,r)$'.
This means that in $\varphi^{B(x,r)}$ all variables range over $B(x,r)$ only.
(This can be achieved by inductively applying the substitutions
$[\forall y : \psi(y)]^{B(x,r)} := \forall y : ((\dist(x,y) \leq r ) \Rightarrow \psi^{B(x,r)}(y))$ and
$[\exists y : \psi(y)]^{B(x,r)} := \exists y : (\dist(x,y) \leq r) \wedge \psi^{B(x,r)}(y)$.)
A {\em basic local} sentence is a sentence of the form

\begin{equation}\label{eq:basiclocal}
\exists x_1, \dots, x_n : \left(\bigwedge_{1\leq i \leq n} \psi^{B(x_i, \ell)}(x_i) \right) \wedge
\left(\bigwedge_{1 \leq i < j \leq n} \dist(x_i, x_j ) > 2\ell \right),
\end{equation}

\noindent
where $\psi(x)$ is a $\FO$-formula whose only free variable is $x$ and $\ell$ is a number.
A {\em local sentence} is a boolean combination of basic local sentences.
The following theorem captures the intuition that first order sentences in some sense can only
capture local properties.
It will help us to shorten some proofs in the sequel.
Besides in~\cite{Gaifman1982}, a proof can for instance be found
in~\cite[Section~2.5]{EbbinghausFlum}.

\begin{theorem}[Gaifman's theorem,~\cite{Gaifman1982}]
\label{thm:gaifman}
Every first order sentence is logically equivalent to a local sentence.
\end{theorem}

\noindent
Here ``$\varphi$ is logically equivalent to $\psi$'' of course means 
that $G \models \varphi$ if and only $G \models \psi$ (for every $G$).

\subsection{Preliminaries on minor-closed classes\label{sec:preliminariesonminorclosedclasses}}

In this section we introduce some  notions and results on minor-closed classes of graphs
that we need in later arguments.
We say that a class of graphs $\Gcal$ is {\em decomposable}
if $G \in \Gcal$ if and only if every component of $G$ is in $\Gcal$.
We say that $\Gcal$ is {\em addable} if it is decomposable, and
closed under adding an edge between vertices in distinct components.
Let us mention (although we shall not use this anywhere in the paper) that
a minor-closed class is addable if and only if it can be characterised by a list of
excluded minors that are all $2$-connected \cite[p.~1]{McDiarmid09}.

Throughout this paper, $\Gcal$ denotes a minor-closed class of graphs,
$\Ccal$ the set of all {\em connected} graphs in $\Gcal$, $\Gcal_n$ the graphs of~$\Gcal$
on vertex set $\{1,\dots,n\}$, $\Ccal_n$ the connected elements of $\Gcal_n$,
and $\Ucal\Gcal$ denotes the unlabelled class corresponding to $\Gcal$, i.e.,
the set of all isomorphism classes of graphs in $\Gcal$.
We define $\Ucal\Ccal, \Ucal\Gcal_n, \Ucal\Ccal_n$ similarly.
The {\em exponential generating function of $\mathcal{G}$} is defined by
\[ G(z) := \sum_{n=0}^\infty |\Gcal_n|\frac{z^n}{n!}, \]
\noindent
and similarly $C(z) = \sum_{n=1}^\infty |\Ccal_n|\frac{z^n}{n!}$.
(Note that by convention the ``empty graph" is not considered connected, so that
$|\Gcal_0| = 1$ and $|\Ccal_0| = 0$.)
If $\Gcal$ is decomposable then it can be seen that $G(z)$ and $C(z)$
are related by the {\em exponential formula}
(see \cite[Lemma~2.1~(i)]{ComptonI} or \cite[Chapter~II]{FlajoletSedgewick}).

\begin{equation}\label{eq:rhodef}
G(z) = \exp(C(z)).
\end{equation}

\noindent
The {\em radius of convergence} of $G(z)$ will always be denoted by $\rho$.
(We remark that, for a decomposable graph class, the exponential formula implies that 
$C(.)$ and $G(.)$ have the same radius of convergence and that 
$C(\rho)$ is finite if and only if $G(\rho)$ is.)

Note that we have

\begin{equation}\label{eq:radiusofconvergence}
 \rho = \left( \limsup_{n\to\infty} \left(\frac{|\Gcal_n|}{n!}\right)^{1/n} \right)^{-1}.
\end{equation}

\noindent

By a result of Norine, Seymour, Thomas and Wollan~\cite{norineseymourthomaswollan},
we know that $\rho > 0$ for every minor-closed class other than the class of all graphs.
More detailed information on which values $\rho$ can assume for minor-closed classes
was obtained by Bernardi, Noy and Welsh~\cite{bernardinoywelsh}.
Amongst other things, they showed that the radius of convergence is infinite
if and only if $\Gcal$ does not contain every path; and that otherwise, if $\Gcal$ contains
all paths, then $\rho \leq 1$.
An arbitrary class of graphs $\Gcal$ is said to be {\em smooth}~if
\begin{equation}\label{eq:conditionforsmoothness}
 \lim_{n\to\infty} \frac{n |\Gcal_{n-1}| }{ |\Gcal_n|}\qquad \text{ exists and is finite. }
\end{equation}

\noindent
(In case this limit does exist then it must in fact equal $\rho$.)
Smoothness turns out to be a key property in many proofs on enumerative and
probabilistic aspects of graphs from minor-closed classes.
The following result was proved by McDiarmid in~\cite{McDiarmid09}.
The statement below combines Theorem 1.2 and Lemma 2.4 of~\cite{McDiarmid09}.

\begin{theorem}[\cite{McDiarmid09}]\label{thm:addissmooth}
Let $\Gcal$ be an addable, minor-closed class, and let $\Ccal \subseteq \Gcal$ be the
corresponding class of connected graphs.
Then $\Ccal$ and $\Gcal$ are both smooth.
\end{theorem}

\noindent
A crucial object in the literature on random graphs from minor closed classes
is the Boltzmann-Poisson random graph, which we define next.

\begin{definition}[Boltzmann--Poisson random graph]\label{def:Boltz}
Let $\Gcal$ be a decomposable class of graphs, and let $\rho$ be the radius of convergence of
its exponential generating function $G(z)$.
If $G(\rho) < \infty$ then the {\em Boltzmann--Poisson random graph}
corresponding to $\Gcal$ is the unlabelled random graph $R$ satisfying:

\begin{equation}\label{eq:Boltzeq}
\Pee( R = H ) = \frac{1}{G(\rho)} \cdot \frac{\rho^{v(H)}}{\aut(H)}
\qquad \text{ for all $H \in \Ucal\Gcal$. }
\end{equation}
\end{definition}

\noindent
Here $\aut(H)$ denotes the number of automorphisms of $H$, where the number of automorphisms of the empty graph is taken to 
be one.
It can for instance been seen from Burnside's lemma that~\eqref{eq:Boltzeq} indeed defines
a probability distribution taking values in $\Ucal\Gcal$. (Alternatively, see Theorem 1.3 of~\cite{McDiarmid09}.)
The paper~\cite{McDiarmid09} also establishes the following result, which we state as a separate
lemma for future convenience.

\begin{lemma}\label{lem:Boltzcomp}
Let $\Gcal, \rho$ and $R$ be as in Definition~\ref{def:Boltz}. Let
$H_1, \dots, H_k \in \Gcal$ be non-isomorphic connected graphs from $\Gcal$ and
let  $Z_i$ denote the number of components of $R$ that are isomorphic to $H_i$.
Then $Z_1, \dots, Z_k$ are independent Poisson random variables with
means $\Ee Z_i = \rho^{v(H_i)} / \aut(H_i)$.
\end{lemma}

The following is a slight rewording of Theorem 1.5 in~\cite{McDiarmid09},
where something stronger is proved.

\begin{theorem}[\cite{McDiarmid09}]\label{thm:addfrag}
Let $\Gcal$ be an addable minor-closed class other than the class of all graphs,
let $\rho$ be its radius of convergence and let $G_n \in_u \Gcal_n$.
Then $G(\rho) < \infty$ and if $F_n$ denotes the isomorphism
class of $\Frag(G_n)$ (so $F_n$ is the {\em unlabelled} version of $\Frag(G_n)$)
then $F_n \to_{\text{TV}} R$, where
$R$ is the Boltzmann--Poisson random graph associated with $\Gcal$.
\end{theorem}

This powerful result has several useful immediate corollaries,
as pointed out in~\cite{McDiarmid09}. For instance, it follows that,
if $\Gcal$ is addable and minor-closed, then $|\Biggg(G_n)| = n - O(1)$ w.h.p., and
\begin{equation}\label{eq:conneq}
\lim_{n\to\infty}\Pee( G_n \text{ is connected } ) =
\lim_{n\to\infty}\Pee( \Frag(G_n) = \emptyset ) = \Pee( R = \emptyset ) = \frac{1}{G(\rho)}.
\end{equation}

\noindent

McDiarmid, Steger and Welsh~\cite{McDiarmidStegerWelsh06} remarked  that
for the case of forests, the asymptotic probability of being connected
is $1/G(\rho) = \upe^{-1/2}$, and they also conjectured
that this is the smallest possible value over all weakly addable graph classes (a class of graphs $\Gcal$ is weakly addable if
adding an edge between distinct components of a graph in $\Gcal$ always produces another graph in $\Gcal$).
This conjecture was proved under some conditions which 
are met by addable, minor closed classes, by two independent teams:
Addario-Berry, McDiarmid and Reed~\cite{addarioberrymcdiarmidreed},
and Kang and Panagiotou~\cite[Theorem~1.1]{KangPanagiotou}. Even more recently (since we submitted the present paper) the 
conjecture appears to have been settled in the affirmative by Chapuy and Perarnau~\cite{ChapuyPerarnau}.
A corollary of the result of~\cite{addarioberrymcdiarmidreed, KangPanagiotou}  is the following.

\begin{theorem}[\cite{addarioberrymcdiarmidreed, KangPanagiotou}]\label{thm:addario}
If $\Gcal$ is an addable, minor-closed class of graphs then
$G(\rho) \leq \sqrt{\upe}$.
\end{theorem}

Let $H$ be a connected graph with a distinguished vertex $r$, the `root'. We say that $G$ contains a {\em pendant copy of $H$} if $G$ contains an induced subgraph isomorphic to $H$, and there is exactly one edge between this copy of $H$ and the rest of the graph, and this edge is incident with the root $r$.
McDiarmid \cite[Theorem 1.7]{McDiarmid09}
proved the following remarkable result:

\begin{theorem}[\cite{McDiarmid09}]\label{thm:pendantcopyadd}
Let $\Gcal$ be an addable, minor-closed class and $G_n \in_u \Gcal_n$.
Let $H \in \Gcal$ be any fixed, connected (rooted) graph.
Then, w.h.p., $G_n$ contains $\Omega(n)$-many pendant copies of $H$.
\end{theorem}

While not explicitly remarked in~\cite{McDiarmid09},
the result carries over to the
random connected graph from~$\Gcal$.

\begin{corollary}\label{cor:pendantcopyadd}
Let $\Gcal$ be an addable, minor-closed class and let $C_n \in_u\Ccal_n$ be the
random {\em connected} graph from $\Gcal$.
If $H \in \Gcal$ is any fixed, connected (rooted) graph then, w.h.p.,
$C_n$ contains $\Omega(n)$-many pendant copies of $H$.
\end{corollary}

\begin{proof}
By Theorem~\ref{thm:pendantcopyadd}, there is a constant $\alpha > 0$  such that
$\Pee( E_n ) = 1-o(1)$, where $E_n$ denotes the event that $G_n$ contains at least $\alpha n$
pendant copies of $H$. Let $F_n$ denote the event that $C_n$ contains at least
$\alpha n$ pendant copies of $H$, and let $A_n$ denote the event that $G_n$ is connected.
Aiming for a contradiction,  let us suppose that
$\liminf_{n\to\infty} \Pee( F_n ) = \beta < 1$.
Observe that if we condition on $A_n$, we find that
$G_n$ is distributed like $C_n$ ($G_n$ is now chosen uniformly at random from all
connected graphs from $\Gcal_n$).
Writing $A_n^c$ for the complement of $A_n$, we see that:
\[ \begin{array}{rcl}
\displaystyle
\liminf_{n\to\infty} \Pee( E_n )
& = &
\displaystyle \liminf_{n\to\infty}  {\Big[} \Pee( F_n ) \cdot \Pee( A_n )
+ \Pee( E_n | A_n^c ) \cdot (1 - \Pee( A_n ) ) {\Big]}  \\
& \leq &
\displaystyle \liminf_{n\to\infty} {\Big[} \Pee( F_n ) \cdot \Pee( A_n )
+ 1 \cdot (1 - \Pee( A_n  ) ) {\Big]} \\
& = &
\beta \cdot \frac{1}{G(\rho)} + (1-\frac{1}{G(\rho)}),
\end{array} \]
\noindent
using~\eqref{eq:conneq} for the last line. But this last expression is $< 1$,
a contradiction. Hence we must have $\Pee( F_n ) = 1 - o(1)$, as required.
\end{proof}

In the paper~\cite{McDiarmid08}, McDiarmid proved a result analogous
to Theorem~\ref{thm:addfrag} above for the class $\Gcal_S$ of all graphs embeddable
on some fixed surface $S$
under the additional assumption that $\Gcal_S$ is smooth. 
That $\Gcal_S$ \emph{is} indeed smooth for every surface $S$ was later established by
Bender, Canfield and Richmond~\cite{bendercanfieldrichmond}.
See also~\cite{BenderGao,chapuyfusygimenezmoharnoy}, where more detailed asymptotic information
is derived for the number of graphs on $n$ vertices from $\Gcal_S$.
By combining \cite[Theorem~3.3]{McDiarmid08} with \cite[Theorem~2]{bendercanfieldrichmond}
we obtain:

\begin{theorem}[\cite{McDiarmid08, bendercanfieldrichmond}]\label{thm:surfacefrag}
Let $S$ be any surface, let $\Gcal$ be the class of all graphs embeddable on $S$
and let $G_n \in_u \Gcal_n$. If $F_n$ denotes the isomorphism class of $\Frag(G_n)$
then $F_n \to_{\text{TV}} R$, where $R$ is the Boltzmann--Poisson random graph associated
with the class of {\em planar} graphs $\Pcal$.
\end{theorem}

Let us stress that the fragment in this last case follows the
Boltzmann--Poisson distribution associated with the class of {\em planar} graphs.
Hence the asymptotic distribution of the fragment is independent of the choice of
surface $S$. Of course, the remarks following Theorem~\ref{thm:addfrag} also apply to
the case of graphs on surfaces, where $\rho, G(\rho)$ are the values for the class
of planar graphs.

Without having to assume smoothness, McDiarmid~\cite{McDiarmid08} was able to
prove the analogue of Theorem~\ref{thm:pendantcopyadd} for graphs on surfaces.

\begin{theorem}[\cite{McDiarmid08}]\label{thm:pendantcopysurfaces}
Let $S$ be a fixed surface, let $\Gcal$ be the class of
all graphs embeddable on $S$ and let $G_n \in_u \Gcal_n$.
Let $H$ be any fixed, connected (rooted) {\em planar} graph.
Then, w.h.p., $G_n$ contains linearly many pendant copies of $H$.
\end{theorem}

\noindent
A verbatim repeat of the proof of Corollary~\ref{cor:pendantcopyadd} now also yields:

\begin{corollary}\label{cor:pendantcopysurfaces}
Let $S$ be a fixed surface, let $\Gcal$ be the class of
all graphs embeddable on $S$ and let and $C_n \in_u \Ccal_n$
be the random {\em connected} graph from $\Gcal$.
Let $H$ be any fixed, connected (rooted) {\em planar} graph.
Then, w.h.p., $C_n$ contains linearly many pendant copies of $H$.
\end{corollary}

We also need another powerful result showing that the random graph embeddable
on a fixed surface is locally planar in the sense given by the next theorem.
It was essentially proved in~\cite{chapuyfusygimenezmoharnoy}, but not stated there explicitly.
For this reason we give a short sketch of
how to extract a proof from the results in \cite{chapuyfusygimenezmoharnoy}.

\begin{theorem}
\label{thm:facewidth}
Let $S$ be any fixed surface, let $\Gcal$ be the class of all graphs embeddable on $S$,
let $G_n \in_u \Gcal_n$, and let $r \in \eN$ be fixed. Then w.h.p.~$B_{G_n}(v, r)$
is planar for all $v \in V(G_n)$.
\end{theorem}

\noindent
{\bf Proof sketch:}
Let $M$ be an embedding of a graph $G$ on a surface $S$.
The face-width $\fw(M)$ of $M$ is the minimum number of intersections of $M$ with a simple non-contractible curve $C$ on $S$. It is easy to see that this minimum is achieved when $C$ meets $M$ only at vertices of $G$.
Notice that if $\fw(M) \ge 2r$, then all the balls in $G$ of radius $r$ are planar.

Fix a surface $S$. From the results in \cite{chapuyfusygimenezmoharnoy} it follows that for any fixed $k$, a random graph that can be embedded in $S$ has an embedding in $S$ with face-width at least $k$ w.h.p. This is first established for 3-connected graphs embeddable in $S$; see \cite[Lemma 4.2]{chapuyfusygimenezmoharnoy}. It is also proved that w.h.p. a random connected graph $G$ embeddable in $S$ has a unique 3-connected component $T$ of linear size (whose genus is the genus of $S$), and the remaining 3- and 2-connected components are planar. The component $T$ is not uniform among all 3-connected graphs with the same number of vertices, since it carries a weight on the edges. But since
Lemma 4.2 in~\cite{chapuyfusygimenezmoharnoy} holds for weighted graphs, the component $M$ has large face-width w.h.p. Since the remaining components are planar, this also applies to $G$.
By Theorem 2.14, this implies the same result for arbitrary graphs. Analogous results were obtained independently in \cite{BenderGao}.
\noproof 
\bigskip

Again a nearly verbatim repeat of the proof of Corollary~\ref{cor:pendantcopyadd} shows:

\begin{corollary}\label{cor:facewidth}
Let $S$ be any surface, let $\Gcal$ be the class of all graphs embeddable on $S$,
let $C_n \in_u \Ccal_n$ be the random connected graph from $\Gcal$, and let $r \in \eN$
be fixed. Then w.h.p.~$B_{C_n}(v, r)$ is planar for all $v \in V(C_n)$.
\end{corollary}

As evidenced by the results we have listed here,
random graphs embeddable on a fixed surface behave rather similarly
to random planar graphs. In particular, despite not being an addable class,
the size of their largest component essentially behaves like the largest
component of a random planar graph (as the number
of vertices not in the largest component is described by the Boltzmann-Poisson random graph for
planar graphs -- cf.~Theorem~\ref{thm:surfacefrag}).

In general, however, non-addable minor-closed graph classes can display
a very different behaviour: see, for example, the recent paper \cite{arXiv:1303.3836v1},
where several non-addable graph classes are analysed in detail.
In particular, the largest component can happen to be sublinear w.h.p.,
as opposed to $n-O(1)$ w.h.p. for the special non-addable class of graphs on a fixed surface.

Therefore, we cannot expect a result like Theorem~\ref{thm:addfrag} to hold
for general smooth, decomposable, minor-closed classes.
Using another result of McDiarmid, we are however able to recover a Poisson
law for component counts under relatively general conditions.
The following lemma is a special case of Lemma 4.2 in~\cite{McDiarmid09}.

\begin{lemma}\label{lem:decompcount}
Let $\Gcal$ be a smooth, decomposable, minor-closed class of graphs
and let $G_n \in_u \Gcal_n$. Let $H_1, \dots, H_k \in \Gcal$ be non-isomorphic,
fixed, connected graphs, and let $N_i$ denote the number of components of $G_n$
isomorphic to $H_i$.Then
\[ (N_1,\dots, N_k) \to_{\text{TV}} (Z_1,\dots,Z_k), \]
\noindent
where the $Z_i$ are independent Poisson random variables with means
$\Ee Z_i = \rho^{v(H_i)}/\aut(H_i)$, and $\rho$ is the radius of convergence
of the exponential generating function for $\Gcal$.\hfill $\blacksquare$
\end{lemma}

Two examples of minor-closed classes that are decomposable, but not addable,
are forests of paths and forests of caterpillars.
Very recently Bousquet-M{\'e}lou and Weller~\cite{arXiv:1303.3836v1} derived precise
 asymptotics for the numbers of labelled forests of paths (resp.~caterpillars).
As a direct corollary of their Propositions~23 and 26 we have:

\begin{theorem}[\cite{arXiv:1303.3836v1}]\label{thm:kerstin}
The classes $\left\{ \text{forests of paths}\right\}$
and $\left\{\text{forests of caterpillars}\right\}$ are both smooth.
\hfill
\end{theorem}

Bousquet-M{\'e}lou and Weller~\cite{arXiv:1303.3836v1} also analysed the class of
all graphs whose components have order at most $t$ (fixed).
This is clearly a minor-closed class that is decomposable, but not addable.
It is in fact not smooth, but a similar property follows from
Proposition 20 in~\cite{arXiv:1303.3836v1}.

\begin{corollary}[\cite{arXiv:1303.3836v1}]\label{cor:kerstinbounded}
Let $t \in \eN$ be fixed and let $\Gcal$ be the class of all graphs whose
components have at most $t$ vertices. There is a constant $c = c(t)$ such that
\[ \frac{n |\Gcal_{n-1}|}{|\Gcal_n|} \sim c \cdot n^{1/t}.  \]
\end{corollary}

\section{The logical limit laws for the addable and {surface} case\label{sec:lawproofs}}

\subsection{The \MSO-zero-one law for addable classes}

The main logical ingredient we need is the following theorem
that is inspired by a construction of McColm~\cite{McColm2002}.

\begin{theorem}\label{thm:MSOMk}
Let $\Gcal$ be an addable, minor-closed class of graphs.
For every $k \in \eN$, there exists a connected (rooted) graph
$M_k \in \Gcal$ with the following property : for every connected $G \in \Gcal$
that contains a pendant copy of $M_k$, it holds that $G \equivMSO_k M_k$.
\end{theorem}

Before starting the proof we need to introduce some more notation.
Recall that a rooted graph is a graph $G$ with a distinguished vertex $r \in V(G)$.
If $G, H$ are two rooted graphs then we say that a third graph $I$ is the result of
{\em identifying their roots} if $I$ can be obtained as follows.
Without loss of generality we can assume $V(G) = \{r\} \cup A, V(H) = \{r\} \cup B$
where $r$ is the root in both graphs and $A, B$ are disjoint.
Then $I = (V(G)\cup V(H), E(G)\cup E(H))$ is the graph we get by `identifying the roots'.

The {\em rooted \EF-game} $\EHRRMSO_k(G, H)$
is played on two rooted graphs $G, H$ with roots $r_G$ and $r_H$
in the same way as the unrooted version.
The only difference is that for Duplicator to win, at the end of the game the following additional conditions
have to be met in addition to conditions {\bf 1)}, {\bf 2)}, {\bf 3)} from the description
of the Ehrenfeucht--Fra{\"\i}ss{\'e}-game in Section~\ref{sec:logicalpreliminaries}:
{\bf 4)} $x_i = r_G$ if and only if $y_i = r_H$,
{\bf 5)} $x_i \sim r_G$ if and only if $y_i \sim r_H$,
and {\bf 6)} $r_G \in X_i$ if and only if $r_H \in Y_i$.
We can view it as the ordinary game with one additional move, where the first move
of both players is predetermined to be a vertex-move selecting the root.
We write $G \equivRMSO_k H$ if $\EHRRMSO_k(G, H)$  is a win for Duplicator.
Note that $\equivRMSO_k$ is an equivalence relation with finitely many equivalence classes
(using that $\equivMSO_{k+1}$ has finitely many equivalence classes and the previous remark).

The next two lemmas are the natural analogues of
Lemmas~\ref{lem:disjointunion} and~\ref{lem:lotsofcopies}
for the rooted \MSO-\EF-game.
For completeness we include self-contained proofs.

\begin{lemma}\label{lem:ident}
Suppose that $G_1\equivRMSO_k H_1, G_2 \equivRMSO H_2$, and let
$G$ be obtained by identifying the roots of $G_1, G_2$ and let
$H$ be obtained by identifying the roots of $H_1, H_2$.
Then $G \equivRMSO_k H$.
\end{lemma}

\begin{proof}
It is convenient to assume that $V(G_1)$ and $V(G_2)$ have exactly one element in common,
the root $r_G$ of $G$, and to identify $G_1, G_2$ with the copies in $G$. And similarly
for $H_1, H_2$ and $H$.

The winning strategy for Duplicator is as follows.
If Spoiler does a vertex move, say he selects $v$ of  $G_\ell$ with $\ell \in \{1,2\}$, then Duplicator responds by selecting a vertex of $H_\ell$ according to his winning strategy for $\EHRRMSO_k(G_\ell, H_\ell)$.
Note that no confusion can arise if Spoiler selects the root of either graph since
Duplicator must then select the root of the other graph (otherwise he loses immediately).
Similarly, Duplicator never selects the root if Spoiler did not also select the root.

If Spoiler does a set move then Duplicator responds as follows.
Suppose Spoiler selected $X \subseteq V(G)$, and let us write
$X_\ell:= X\cap V(G_\ell)$ for $\ell =1,2$.
Then Duplicator selects a set $Y_\ell \subseteq V(H_\ell)$ for each $\ell \in \{1,2\}$, according to the winning strategy for $\EHRRMSO_k(G_\ell, H_\ell)$, and then sets $Y := Y_1\cup Y_2$ as his response to Spoiler's move.
Again, no confusion can arise because of the presence or not of the root in $X$.
If Spoiler selects a subset $Y \subseteq V(H)$ then Duplicator responds analogously.
This is a winning strategy for Duplicator as every
edge of $G$ is either an edge of $G_1$ or of $G_2$ and every vertex of $G$ other than the root is either a vertex of $G_1$ or of $G_2$; and similarly for $H$.
\end{proof}

\begin{lemma}\label{lem:ab}
For every $k \in \eN$ there is an $a = a(k)$ such that the following holds.
For every rooted graph $G$ and every $b \geq a$, if
$A$ is obtained from $a$ copies of $G$ and identifying the roots,
and $B$ is obtained from $b$ copies of $G$ and identifying the roots,
then $A \equivRMSO_k B$.
\end{lemma}
\begin{proof}
Before proving the full statement, we prove a seemingly weaker statement.\\

\medskip 

\noindent
\textbf{Claim}. Let $G$ and $a \geq 2^{k\cdot v(G)}$ be arbitrary, and
let a rooted graph $A$ be obtained by identifying the roots of $a+1$ copies of $G$
and a rooted graph $B$ by identifying $a$ copies of $G$. Then $A \equivRMSO_k B$.

\begin{proofof}{Claim}
To prove the claim, let $A_1, \dots, A_{a+1}$ be the copies of $G$ that make up $A$, and let
$B_1, \dots, B_a$ be the copies of $G$ that make up $B$.
For any $i,j$ and $v \in V(G)$ we denote by $v_i^A$ (resp. $v_j^B$)
the unique copy of vertex $v$ inside the copy $A_i$ of $G$
(resp. inside the copy $B_j$ of $G$).

Note that because of the demands {\bf 4)} and {\bf 5)}
for Duplicator's win, Spoiler basically `wastes a move'
when selecting the root of either graph in a vertex-move, since the root
already behaves like a marked vertex.
So if Spoiler can win $\EHRRMSO_k(A, B)$,
then winning is possible without ever making a vertex-move selecting the root of either graph.
In the sequel we thus assume that Spoiler's vertex-moves never selects the root.

We describe the situation after any move of the game $\EHRRMSO_k(A, B)$
by the graphs $A, B$, together with some `vertex-move-marks' and `set-move-marks'
on their vertices, which record the number of the move when the respective vertex
or vertex-set was selected, and whether it was a vertex- or set-move.
(The information whether it was Spoiler or Duplicator who made a choice is not recorded.)

For any move $0 \leq k' \leq k$,
we say that $A_i$ and $B_j$ are {\em marked identically} to each other
(after move $k'$)
if for every move $k'' \leq k'$ the following holds:
if $k''$ was a vertex-move then $v_i^A$ was marked
if and only if
$v_j^B$ was marked (for all $v \in V(G)$)
and if $k''$ was a set-move then
$v_i^A$ was in the selected subset of $V(A)$
if and only if
$v_j^B$ was in the selected subset of $V(B)$ (for all $v \in V(G)$).
Similarly, we also speak of $A_i$ being marked identically to $A_j$
if and only if no vertex-move occurred on either one, and every set-move
until now selected the same subset from both.

For any $0 \leq k' \leq k$, let us say that the situation of the game after move $k'$
is {\em good} (for Duplicator) if the following holds:

\begin{itemize}
\item $A_i$ and $B_i$ are marked identically for all $1\leq i \leq a$;
\item no vertex of $A_{a+1}$ was marked by a vertex-move until now;
\item there are at least $2^{(k-k')\cdot v(G)}$ indices $1\leq i\leq a$
such that $B_i$ is marked identically to $A_{a+1}$.
\end{itemize}

\noindent
Observe that if after move $k'=k$ the situation is still good, Duplicator has
won the game. Our aim is to show by induction that Duplicator can indeed achieve
this situation, up to a relabelling of $A_1, \dots, A_{a+1}$ and a relabelling of $B_1, \dots, B_a$.
Clearly, the situation is good at the beginning of the game, when no moves have been
played yet, corresponding to $k'=0$.

Now assume that after move $k'< k$, the situation is good. We show that,
no matter what Spoiler does in move $k'+1$, Duplicator can respond in such a way
that after move $k'+1$ the situation is still good, possibly after some relabelling.

To see this, let us first suppose that Spoiler does a vertex-move:
if this marks the vertex $v_i^A \in V(A_i)$ for some $v\in V(G)$ and $1\leq i \leq a$,
then Duplicator
simply responds by marking the vertex $v_i^B \in V(B_i)$. Observe that now we are
still in a good situation: $A_i$ and $B_i$ are marked identically for $i=1,\dots, a$,
and $A_{a+1}$ has none of its vertices marked by a vertex-move and is marked
identically to at least $2^{(k-k')\cdot v(G)} -1 \geq 2^{(k-k'-1)\cdot v(G)}$
of the $B_i$. Similarly, if Spoiler chooses to mark the vertex $v_i^B \in V(B_i)$
for some $v\in V(G)$ and $1\leq i \leq a$, then Duplicator can again respond by
marking $v_i^A$ and we are still good.
Assume thus that Spoiler marked a vertex of $A_{a+1}$. Since there are
$2^{(k-k')\cdot v(G)} > 1$ indices $i$ such that $A_i$ is marked identically
to $A_{a+1}$ we can just relabel $A_1,\dots, A_{a+1}$ and arrive at the situation
where Spoiler chose a vertex in an $A_i$ with $i \leq a$.
But then the induction is again done, by the above. This completes the
case when Spoiler does a vertex-move.

We now consider set-moves. In the rest of the proof, let $I \subseteq [a]$
be the set of those indices $i\in[a]$ for which $B_i$ is marked identically to $A_{a+1}$,
after move $k'$.

First suppose that Spoiler in move $k'+1$ selected a subset $X \subseteq V(A)$.
Observe that, since $G$ has $2^{v(G)}$ subsets in total, the
set $I \cup \{a+1\}$ is partitioned into $L\leq 2^{v(G)}$ subsets
$I_1, \dots, I_L$ such that, for any $1\leq \ell\leq L$,
if $i, j \in I_\ell$ then $A_i, A_j$ are marked identically
after Spoiler's $(k'+1)$-st move.
There must be some $\ell$ such that
$|I_\ell| \geq (|I|+1)/2^{v(G)} > 2^{(k-k'-1)\cdot v(G)}$.
Relabelling if necessary, we can assume without loss of generality that
$a+1 \in I_{\ell}$ for such an index $\ell$.
(So in particular $A_{a+1}$ is marked identically to at least
$2^{(k-k'-1)\cdot v(G)}$ of the other $A_i$.)
Duplicators response is simply to select $Y \subseteq V(B)$
according to the rule that $v_i^B \in Y$ if and only if $v_i^A \in X$ -- 
for all $1\leq i \leq a$ and all $v \in V(G)$. (This is of course done after any possible relabelling
of $A_1,\dots A_{a+1}$ and $B_1,\dots, B_a$.)
Note that no confusion can arise because of the root.
It is easily seen that this way, the situation is still good after move $k'+1$.

Suppose then that Spoiler selected a subset $Y \subseteq V(B)$.
Now $I$ is partitioned into $L\leq 2^{v(G)}$ subsets
$I_1, \dots, I_L$ such that, for any $1\leq \ell \leq L$,
the sets $B_i, B_j$ are marked identically after Spoiler's move whenever $i,j \in I_\ell$.
There is some $1\leq\ell\leq L$ such that $|I_\ell| \geq |I| / 2^{v(G)} \geq
2^{(k-k'-1)\cdot v(G)}$.
We fix such an $\ell$ and an $i_0 \in I_\ell$.
Duplicators response is to select $X \subseteq V(A)$ according to the
rules that $v_i^A \in X$ if and only if $v_i^B \in Y$
(for all $1\leq i \leq a$ and all $v \in V(G)$) and that
$v_{a+1}^A \in X$ if and only if $v_{i_0}^B \in Y$
(for all $v \in V(G)$).
Again it is easily seen that the situation is still good after move $k'+1$.

We have seen that indeed, no matter which move Spoiler chooses to make,
Duplicator can always respond in such a way that the situation will stay good.
Hence $A\equivRMSO_k B$, which completes the proof of the claim.
\end{proofof}

Having proved the claim, we are ready for finish the proof of the lemma.
Observe that, by repeated applications of the claim,
we also have that $A \equivRMSO_k B$ for all $a, b \geq 2^{k\cdot v(G)}$
if $A$ is obtained by identifying the roots of $a$ copies of $G$ and
$B$ by identifying the roots of $b$ copies.

Since there are finitely many equivalence classes for $\equivRMSO_k$, there is
a finite list of graphs $H_1, \dots, H_\ell$ such that every graph is equivalent
to one of them. Let us set
\[ a = a(k) := \max\left( 2^{k \cdot v(H_1)}, \dots, 2^{k\cdot v(H_\ell)} \right). \]
\noindent
Let $G$ be an arbitrary graph, and $b \geq a$ be arbitrary.
There is some $1\leq i\leq \ell$ such that $G \equivRMSO_k H_i$.
Let $A, A'$ be obtained by identifying the roots of $a$ copies of $G$ resp.~$H_i$,
and let $B, B'$ be obtained by identifying the roots of $b$ copies of $G$ resp.~$H_i$.
By Lemma~\ref{lem:ident}, we also have $A \equivRMSO_k A'$ and
$B \equivRMSO_k B'$. By the claim and the observation we made immediately after
its proof, we have $A' \equivRMSO_k B'$. Hence also $A \equivRMSO_k B$.
This proves that our choice of $a(k)$ indeed works for {\em every} graph $G$,
and concludes the proof of Lemma~\ref{lem:ab}.
\end{proof}

\begin{proofof}{Theorem~\ref{thm:MSOMk}}
The construction of $M_k$ is as follows.
Recall that $\Ccal$ denotes the set of all connected elements of $\Gcal$.
Let $\mathrm{r}\Gcal$ be the set of all rooted graphs corresponding to $\Gcal$.
That is, for each element $G \in \Gcal$ there are $v(G)$ elements in $\mathrm{r}\Gcal$, one
for each choice of the root. We define $\mathrm{r}\Ccal$ similarly.
As remarked previously, Lemma~\ref{lem:fixednoeq}, despite being about $\equivMSO_k$,
implies that the relation $\equivRMSO_k$ partitions the set of all rooted finite graphs,
and hence in particular $\mathrm{r}\mathcal{C}$, into finitely many equivalence classes.
Hence there exists a finite set of \emph{connected} rooted graphs
$G_1, \dots, G_m \in \mathrm{r}\Ccal$
such that every \emph{connected} rooted graph from $\mathrm{r}\Ccal$ is equivalent
under $\equivRMSO_k$ to one of $G_1, \dots, G_m$.

The graph $M_k$ is now constructed by taking $a$ copies of $G_i$ for each $i=1,\dots,m$,
and identifying their roots, where $a = a(k)$ is as provided by Lemma~\ref{lem:ab}.
Let us remark that $\Gcal$ being addable and minor-closed implies $M_k \in \mathrm{r}\Gcal$.

Let $G \in \Gcal$ be an (unrooted) connected graph that contains at least one
pendant copy of $M_k$. Let us fix one such pendant copy (for notational convenience
we just identify this copy with $M_k$ from now on). It is convenient to root $G$ at the root
$r$ of $M_k$. Certainly, if Duplicator wins $\EHRRMSO_k( G, M_k)$, then he also wins
the unrooted game $\EHRMSO_k( G, M_k )$. We consider the rooted version of the game
in the remainder of the proof.

Let $G_1, \dots, G_m$ denote the rooted graphs
used in the construction of $M_k$, and let
$G_1^1$, $\dots$, $G_1^a$, $\dots$, $G_m^1$, $\dots$, $G_m^a$ be the copies
of $G_1$ until $G_m$ whose roots were identified to create $M_k$. Let $G'$
denote the (rooted, connected) subgraph of $G$ induced by
$(V(G) \setminus V(M_k)) \cup \{r\}$.
That is, to obtain $G'$ we remove all vertices of $M_k$ from $G$ except the root $r$.
Observe that one can view $G$ as consisting of $G',G_1^1,\dotsc,G_m^a$,
identified along their roots.
Also note that $G' \in \Gcal$ since it is a minor of $G$.
Hence, by choice of $G_1,\dots, G_m$, there exists $1 \leq i \leq m$ such that
$G' \equivRMSO_k G_i$. Without loss of generality $i=1$. It follows from
Lemma~\ref{lem:ident} that $G \equivRMSO_k H$, where $H$ is obtained by taking
$a+1$ copies of $G_1$ and $a$ copies of each of $G_2, \dots, G_n$ and identifying
the roots. By Lemma~\ref{lem:ab} together with Lemma~\ref{lem:ident} we also
have $H \equivRMSO_k M_k$. It follows that $G \equivRMSO_k M_k$, as required.
\end{proofof}

With Theorem~\ref{thm:MSOMk} in hand, we are now ready to prove the $\MSO$-zero-one law
for the random connected graph from an addable, minor-closed class.

\begin{proofof}{Theorem~\ref{thm:MSO01add}}
Let $\varphi \in \MSO$ be arbitrary, let $k$ be its quantifier depth and let $M_k$ be as
provided by Theorem~\ref{thm:MSOMk}.
By Corollary~\ref{cor:pendantcopyadd}, w.h.p., $C_n$ has a pendant copy of $M_k$.
Thus, by Theorem~\ref{thm:MSOMk}, w.h.p., $C_n \equivMSO_k M_k$.
In particular, this implies that if $M_k \models \varphi$ then
$\displaystyle \lim_{n\to\infty} \Pee( C_n \models\varphi ) = 1$
and if, on the other hand, $M_k \models \neg\varphi$ then
$\displaystyle \lim_{n\to\infty} \Pee( C_n \models\varphi ) = 0$.
\end{proofof}

\subsection{The \MSO-convergence law for addable classes}

In this section we prove the following more explicit version
of Theorem~\ref{thm:MSOconvadd} above.

\begin{theorem}\label{thm:MSOconvdetail}
Let $\Gcal$ be an addable, minor-closed class of graphs,
let $G_n \in_u\Gcal_n$ and let $R$ be the Boltzmann--Poisson random graph
corresponding to $\Gcal$. For every $\varphi \in \MSO$ there exists a set
$\Fcal = \Fcal(\varphi) \subseteq \Ucal\Gcal$ such that
\[ \lim_{n\to\infty} \Pee( G_n \models \varphi ) = \Pee( R \in \Fcal ). \]
\end{theorem}

\noindent
(We remark that the Boltzmann-Poisson distribution on $\Ucal\Gcal$ is well-defined in this case because of
Theorem~\ref{thm:addario}.)
In the proof of Theorem~\ref{thm:MSOconvdetail} we make use of the following (nearly) trivial consequence
of Theorem~\ref{thm:pendantcopyadd}.

\begin{corollary}\label{cor:Bigpendantadd}
In the situation of Theorem~\ref{thm:MSOconvdetail}: if $H \in \Gcal$ is any fixed,
connected (rooted) graph, then w.h.p.~$\Biggg(G_n)$ contains a pendant copy of $H$.
\end{corollary}

\begin{proof}
By Theorem~\ref{thm:pendantcopyadd}, there exists $\alpha > 0$ such that
$G_n$ contains at least $\alpha n$ pendant copies of $H$ w.h.p.
Let $A_n$ denote the event that $\Biggg(G_n)$ contains at least $\alpha n/2$ pendant copies of $H$, and let
$B_n$ denote the event that $\Frag(G_n)$ contains at least $\alpha n/2$ copies of $H$.
We must have $\Pee( A_n ) + \Pee( B_n ) \geq \Pee( A_n \cup B_n ) = 1 - o(1)$.
Now observe that, for every fixed $K > 0$ we have
\[ \Pee( B_n ) \leq \Pee( v(\Frag(G_n)) > K )
= \Pee( v(F_n) > K )
= \Pee( v(R) > K ) + o(1), \]
\noindent
where $F_n$ denotes the isomorphism class (i.e., unlabelled version) of $\Frag(G_n)$, the first inequality holds for $n$ sufficiently large
and we use Theorem~\ref{thm:addfrag} for the last equality. 
The probability $\Pee( v(R) > K )$ can be made arbitrarily small by choosing $K$ large enough. From this it follows
that $\Pee(B_n) = o(1)$. Hence $\Pee( A_n ) = 1-o(1)$.
\end{proof}

\begin{proofof}{Theorem~\ref{thm:MSOconvdetail}}
Let $\varphi \in\MSO$ be arbitrary, let $k$ be its quantifier depth and let $M_k$ be as provided by Theorem~\ref{thm:MSOMk}.
By Corollary~\ref{cor:Bigpendantadd}, w.h.p., $\Biggg(G_n)$ contains a pendant copy of $M_k$.
Hence, by Theorem~\ref{thm:MSOMk}, we have $\Biggg(G_n) \equivMSO_k M_k$ (w.h.p.).
From this it also follows, using Lemma~\ref{lem:disjointunion}, that
\[ G_n \equivMSO_k M_k \cup \Frag( G_n ) \quad \text{ w.h.p. } \]
\noindent
(where $\cup$ denotes vertex-disjoint union).
In particular we have $\Pee( G_n \models \varphi ) = \Pee( M_k \cup \Frag(G_n) \models \varphi ) + o(1)$.
Now let $\Fcal \subseteq \Ucal\Gcal$ be the set of all unlabelled graphs $H \in\Ucal\Gcal$ such that
$M_k \cup H \models \varphi$.
It follows using Theorem~\ref{thm:addfrag} that:
\[ \lim_{n\to\infty} \Pee( G_n \models \varphi ) =
\lim_{n\to\infty} \Pee( M_k \cup \Frag( G_n ) \models \varphi ) =
\lim_{n\to\infty} \Pee( F_n \in \Fcal ) = \Pee( R \in \Fcal ), \]
\noindent
where $F_n$ again denotes the isomorphism class (i.e., unlabelled version) of $\Frag(G_n)$.
\end{proofof}

\subsection{The \FO-zero-one law for surfaces}

The proof of the \FO-zero-one law for surfaces mimics that of the
\MSO-zero-one law for the addable case.
The main ingredient is the following analogue of Theorem~\ref{thm:MSOMk}.

\begin{lemma}\label{lem:FOFk}
For every $k \in \eN$ there exists $\ell = \ell(k)$ and a connected, rooted, planar graph $L_k$
such that following holds. For every connected graph $G$ such that
\begin{enumerate}
\item\label{itm:FOFk1} the subgraph of $G$ induced by the vertices at distance
at most $\ell$ from $v$ is planar for every $v \in V(G)$,
\item\label{itm:FOFk2} $G$ contains a pendant copy of $L_k$,
\end{enumerate}
it holds that $G \equivFO_k L_k$.
\end{lemma}

\begin{proof}
Recall that, up to logical equivalence, there are only
finitely many $\FO$-sentences of quantifier depth
$\leq k$. For each such sentence we fix an arbitrary
local sentence that is logically equivalent to it
(such a local sentence exists by Gaifman's theorem).
Let $\Lcal $ be the set of local sentences thus obtained and let
$\Bcal = \{\varphi_1, \dots, \varphi_m \}$ be the set of all
basic local sequences that appear in the boolean combinations $\Lcal$.
Let us set $k' := \max(\qd(\varphi_1), \dots, \qd(\varphi_m))$ equal to
the maximum quantifier depth over all these local sentences.
For each $1\leq i \leq m$, we can write
\begin{equation}\label{eq:varphiBL} \varphi_i =
\exists x_1, \dots, x_{n_i} :
\left(\bigwedge_{1\leq a \leq n_i} \psi_i^{B(x_a, \ell_i)}(x_a) \right) \wedge
\left(\bigwedge_{1 \leq a <  b \leq n_i} \dist(x_a, x_b ) > 2\ell_i \right). \end{equation}
\noindent
We now set $\ell = \ell(k) = \max( \ell_1,\dots, \ell_m )$, and we let
$L_k$ be a path of length $1 + 1000 \cdot \max(n_1,\dots, n_m) \cdot \ell$
with a pendant copy of $M_{k'}$ attached to one of its endpoints, where
$M_{k'}$ is as provided by Theorem~\ref{thm:MSOMk} for the class of planar graphs.
We root $L_k$ at the middle vertex of the long path.

Let us now fix an arbitrary $1\leq i \leq m$.
Let us first suppose that the sentence
$\exists x, y : \psi_i^{B(x,\ell_i)}(x) \wedge (\dist(x,y) = \ell_i)$
is not satisfied by any connected planar graph.
Since $L_k$ is planar and every point is at distance exactly $\ell_i$
from some other point, the sentence $\exists x : \psi_i^{B(x,\ell_i)}(x)$
does not hold for $L_k$. Hence $\varphi_i$ cannot hold for $L_k$ either.
Let $G$ be any connected graph with the properties~\ref{itm:FOFk1} and~\ref{itm:FOFk2}
listed in the statement of the lemma.
Observe that, since $G$ contains a pendant copy of $L_k$, for every $x \in V(G)$ there is 
a vertex at distance exactly
$\ell_i$ from $x$. Moreover, for every vertex $x$, the
subgraph $B(x,\ell_i)$ is connected and planar. 
But this shows $\exists x : \psi_i^{B(x_i,\ell_i)}(x)$
cannot hold for $G$, because that would imply that $\exists x, y : \psi_i^{B(x,\ell_i)}(x) \wedge (\dist(x,y) = \ell_i)$ also holds.
Hence $\varphi_i$ does not hold for $G$ either.

Next, let us suppose that the sentence
$\exists x, y : \psi_i^{B(x,\ell_i)}(x) \wedge (\dist(x,y) = \ell_i)$ is satisfied by at
least one connected planar graph $H$.
We construct a graph $G'$ as follows. We take $n_i$ copies of $H$
and join them to an extra point $u$, via edges to their $y$-vertices. We now attach
a pendant copy of $L_k$ to $u$. (See Figure~\ref{fig:Gprime} for a depiction.)

\begin{figure}[H]
\begin{center}
\input{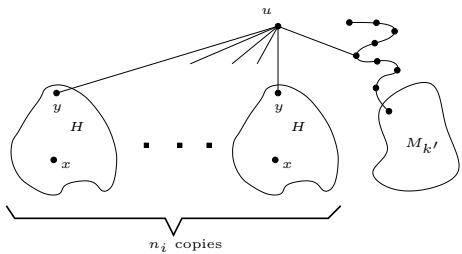}
\end{center}
\caption{The construction of $G'$.\label{fig:Gprime}}
\end{figure}

\noindent
Observe that, by construction, $G' \models \varphi_i$.
Since $G'$ and $L_k$ are planar, by Theorem~\ref{thm:MSOMk} we have
\[ G' \equivMSO_{k'} M_{k'} \equivMSO_{k'} L_k. \]
\noindent
So $L_k \models \varphi_i$ as well.
We now claim that, because of the special form~\eqref{eq:varphiBL} of $\varphi_i$ 
any graph $G$ that contains a pendant copy of $L_k$ will satisfy $\varphi_i$ as well.
To see this, we recall that $L_k$ will be hanging from the middle vertex of the long path.
This gives that, for every vertex of $L_k$ not on the path, its $\ell$-neighbourhood in $G$ will be 
identical to its $\ell$ neighbourhood in $L_k$. The same is true for the endpoint of the long path
that we did not attach $M_{k'}$ to, and for any vertex of the long path within distance
$\ell$ of one of the endpoints of the path.
Finally, we note that for every vertex $v$ on the long path of $L_k$ at distance $>\ell$ from both endpoints we can find a
vertex $u$ of $G$, also on the long path but far enough away from the middle and the endpoints
such that $B_{L_{k}}(v,\ell)$ and $B_G(u,\ell)$ are isomorphic as graphs rooted 
at $v$ resp.~$u$ (both will just be a path of length $2\ell+1$ with root the middle vertex).
So, appealing to the special form~\eqref{eq:varphiBL} of $\varphi_i$, we see that the claim holds. 
That is, $L_k \models \varphi_i$ implies 
$G \models \varphi_i$ for each $G$ that contains a pendant copy of $L_k$ (suspended from the 
middle vertex of the long path).

This proves that if $G$ is any graph satisfying the assumptions of the lemma, then
$G \models \varphi_i$ if and only if $L_k \models \varphi_i$.
Since $1\leq i \leq m$ was arbitrary, and every $\FO$ sentence of quantifier depth
at most $k$ can be written as a Boolean combination of $\varphi_1,\dots, \varphi_m$, it
follows that $G \equivFO_k L_k$ for every $G$ that satisfies the assumptions of the lemma.
\end{proof}

\begin{proofof}{Theorem~\ref{thm:surface01}}
The proof closely follows the structure of the proof of Theorem~\ref{thm:MSO01add}.
Let $\varphi \in \FO$ be arbitrary, let $k$ be its quantifier depth and let
$\ell, L_k$ be as provided by Lemma~\ref{lem:FOFk}.
By Corollary~\ref{cor:facewidth}, w.h.p., every $\ell$-neighbourhood of every vertex
of $C_n$ is planar.
By Corollary~\ref{cor:pendantcopysurfaces}, w.h.p., $C_n$ has a pendant copy of $L_k$.
It thus follows from Lemma~\ref{lem:FOFk} that, w.h.p., $C_n \equivFO_k L_k$.

This implies that if $L_k \models \varphi$ then
$\displaystyle \lim_{n\to\infty}\Pee( C_n\models \varphi ) = 1$.
And if, on the other hand, $L_k \models\neg\varphi$ then
$\displaystyle \lim_{n\to\infty}\Pee( C_n\models \varphi ) = 0$.

Since the graph $L_k$ provided by Lemma~\ref{lem:FOFk} does not
depend on the choice of the surface, the same is true for the value
of the limiting probability.
\end{proofof}

\subsection{The \FO-convergence law for surfaces}

\begin{theorem}\label{thm:surfacesconvdetail}
Fix a surface $S$ and let $\Gcal$ be the class of all graphs embeddable on $S$
and let $G_n \in_u \Gcal_n$.
Let $R$ be the Boltzmann--Poisson random graph corresponding to the class
of {\em planar} graphs $\Pcal$. For every $\varphi \in \FO$ there exists
a set $\Fcal = \Fcal(\varphi) \subseteq \Ucal\Pcal$ such that
\[ \lim_{n\to\infty} \Pee( G_n \models \varphi ) = \Pee( R \in \Fcal ). \]
\noindent
Moreover, $\Fcal(\varphi) \subseteq\Ucal\Pcal$ does not depend on the surface $S$.
\end{theorem}

The proof closely follows that of Theorem~\ref{thm:MSOconvdetail}.
We again separate out a (nearly) trivial consequence of in this
case Theorem~\ref{thm:pendantcopysurfaces}. The proof is completely analogous
to that of Corollary~\ref{cor:Bigpendantadd} and is left to the reader.

\begin{corollary}\label{cor:Bigpendantsurfaces}
If $H$ is any fixed, connected, planar (rooted) graph,
then w.h.p.~$\Biggg(G_n)$ contains a pendant copy of $H$. \hfill $\blacksquare$
\end{corollary}

\begin{proofof}{Theorem~\ref{thm:surfacesconvdetail}}
Let $\varphi \in \FO$ be arbitrary, let $k$ be its quantifier depth and
let $\ell, L_k$ be as provided by Lemma~\ref{lem:FOFk}.
By Theorem~\ref{thm:facewidth}, w.h.p.,
the $\ell$-neighbourhood of every vertex of $G_n$ is planar.
By Corollary~\ref{cor:Bigpendantsurfaces}, w.h.p.,
$\Biggg(G_n)$ has a pendant copy of $L_k$. It follows
by Lemma~\ref{lem:FOFk} that $\Biggg(G_n) \equivFO_k L_k$ (w.h.p.) and hence also,
using Lemma~\ref{lem:disjointunion}:
\[ G_n \equivFO_k L_k \cup \Frag( G_n ) \quad \text{ w.h.p. } \]
\noindent
(Here $\cup$ again denotes vertex-disjoint union).
Let $\Fcal \subseteq \Ucal\Pcal$ denote the set of all unlabelled
{\em planar} graphs $H$ such that $L_k \cup H \models \varphi$, and let $\Fcal'$
denote the set of all unlabelled graphs (not-necessarily planar)
with the same property. Using Theorem~\ref{thm:surfacefrag}, we find that:
\[
\lim_{n\to\infty} \Pee( G_n \models \varphi )
= \lim_{n\to\infty} \Pee( L_k \cup \Frag(G_n) \models \varphi )
= \lim_{n\to\infty} \Pee( F_n \in \Fcal' )
= \Pee( R \in \Fcal' ) = \Pee( R \in \Fcal ), \]
\noindent
where $F_n$ is the isomorphism class of $\Frag(G_n)$, and $R$ is
the Boltzmann--Poisson random graph associated with planar graphs.
(The last equality holds because the distribution of $R$ assigns probability zero to
non-planar graphs.)
It is clear that $\Fcal$ does not depend on the choice of surface, since $L_k$ does not
depend on the surface either.
\end{proofof}

\section{The limiting probabilities\label{sec:limitingprobabilities}}

Throughout this section $\Gcal$ will be an arbitrary addable, minor-closed class.
For notational convenience we shall write
\[ L_{\FO} :=
\left\{ \lim_{n\to\infty} \Pee( G_n \models \varphi ) : \varphi \in \FO \right\}, \quad \quad
L_{\MSO} :=
\left\{ \lim_{n\to\infty} \Pee( G_n \models \varphi ) : \varphi \in \MSO \right\},
\]
where $G_n \in_u \Gcal_n$ as usual.
Here the dependence on the class $\Gcal$ is suppressed in the notations $L_{\FO}, L_{\MSO}$ for readability. 
The class under consideration will 
always be clear from the context.

\subsection{There is always a gap in the middle}

In this section we prove the following lemma about the structure
of the logical limit sets of a general addable minor-closed class,
which together with Theorem~\ref{thm:addario} proves Proposition~\ref{prop:gap}. Notice that $1/G(\rho) \ge 1/\sqrt{e} > 1/2$
by Theorem~\ref{thm:addario}.

\begin{lemma}\label{lem:alwaysagap}
Let $\Gcal$ be an addable, minor-closed class of graphs. Then
$L_\MSO \cap \left(1-\frac{1}{G(\rho)}, \frac{1}{G(\rho)} \right) = \emptyset$.
\end{lemma}
\begin{proof}
Let $\varphi \in \MSO$ be arbitrary, and let $E$ denote the event that $G_n$
is connected. Observe that if we condition on $E$, then $G_n$ is distributed
like $C_n$ ($G_n$ is then chosen uniformly at random from all connected graphs
on $n$ vertices). Since the $\MSO$-zero-one law holds for $C_n$ by
Theorem~\ref{thm:MSO01add} we have either
$\Pee( C_n \models \varphi ) = 1-o(1)$ or $\Pee( C_n \models \varphi ) = o(1)$.
Let us first assume the former is the case. Then
\[ \begin{array}{rcl}
\lim_{n\to\infty} \Pee( G_n \models \varphi )
& = &
\lim_{n\to\infty} \Pee( G_n \models \varphi | E ) \cdot \Pee( E ) + \Pee( G_n \models \varphi | E^c ) \cdot \Pee( E^c ) \\
& \geq  &
\lim_{n\to\infty} \Pee( G_n \models \varphi | E ) \cdot \Pee( E ) \\
& = & \lim_{n\to\infty} \Pee( C_n \models \varphi ) \cdot \Pee( E ) \\
& = & 1 \cdot 1/G(\rho),
\end{array} \]

\noindent
where for the last equality we used the general limit in \eqref{eq:conneq} above.
Now suppose $\Pee( C_n \models \varphi ) = o(1)$.
Then $\lim_{n\to\infty} \Pee( G_n \models \neg\varphi ) \geq 1/G(\rho)$
by the above, hence $\lim_{n\to\infty} \Pee( G_n \models \varphi ) \leq 1 - 1/G(\rho)$.
\end{proof}

\subsection{The closure is a finite union of intervals}

Here we prove the following more detailed version of Theorem~\ref{thm:finintintro}.
Note that both $L_{\mathrm{MSO}}$ and $L_{\mathrm{FO}}$ are countable sets since the
set of MSO-sentences is countable.

\begin{theorem}\label{thm:finint}
Let $\Gcal$ be an addable, minor-closed class of graphs and let
$R$ be the corresponding Boltzmann--Poisson random graph. Then
\[ \clo( L_{\MSO} ) = \clo( L_{\FO} ) = \{ \Pee( R \in \Fcal ) :
\Fcal \subseteq \Ucal\Gcal \}, \]
\noindent
is a finite union of closed intervals.
\end{theorem}

Before starting the proof of Theorem~\ref{thm:finint}, we will derive a number of auxiliary lemmas.
From Theorem~\ref{thm:MSOconvdetail} we see immediately that
$L_{\FO} \subseteq L_{\MSO} \subseteq \{ \Pee( R \in \Fcal ) : \Fcal \subseteq \Ucal\Gcal \}$.
The next lemma shows that $L_{\FO}$ is in fact dense in
$\{ \Pee( R \in \Fcal ) : \Fcal \subseteq \Ucal\Gcal \}$.

\begin{lemma} For every $\Fcal \subseteq \Ucal\Gcal$ and every $\eps > 0$
there is a $\varphi \in\FO$ such that
\[ | \Pee( R \in \Fcal ) - \lim_{n\to\infty} \Pee( G_n \models \varphi ) | \leq \eps . \]
\end{lemma}

\begin{proof}
First note that it in fact suffices to consider only
{\em finite} $\Fcal \subseteq \Ucal\Gcal$.
(To see this, notice that there is always a finite $\Fcal'\subseteq\Fcal$ such that
$\Pee( R \in \Fcal' ) \geq \Pee( R \in \Fcal ) - \eps/2$).
Let us thus assume $\Fcal$ is finite.

Let us pick a $K$ such that $\Pee( v(R) > K ) < \eps$ and let $\Frag_K(G)$
denote the union of all components of $G$ of order at most $K$. Let us observe that,
for every $F \in \Fcal$, the event $\{ \Frag_K(G_n) \cong F \}$ is $\FO$-expressible.
(We simply stipulate, for each of the connected graphs $H \in \Ucal\Ccal$
on at most $K$ vertices, how many components isomorphic to $H$
the random graph $G_n$ should contain.) Since $\Fcal$ is finite,
the event $\{ \Frag_K(G_n) \in \Fcal \} = \bigcup_{F\in\Fcal} \{ \Frag_K(G_n) \cong F \}$
is therefore also $\FO$-expressible. Observe that
\[ \begin{array}{rcl}
\displaystyle \lim_{n\to\infty} \Pee{\Big[} \Frag_K(G_n) \in \Fcal  {\Big]}
& \leq &
\displaystyle \lim_{n\to\infty} \Pee{\Big[} \Frag(G_n) \in \Fcal \text{ or }
v(\Frag(G_n) ) > K {\Big]}  \\
& \leq &
\displaystyle \lim_{n\to\infty} \Pee{\Big[} \Frag(G_n) \in \Fcal {\Big]} + \lim_{n\to\infty}  \Pee{\Big[}  v(\Frag(G_n)) > K {\Big]}  \\
& < &
\Pee( R \in \Fcal ) + \eps,
\end{array} \]

\noindent
where in the last line we used that $\Pee( v(R) > K ) < \eps$ by the choice of $K$.

Similarly,
\[ \begin{array}{rcl}
\displaystyle \lim_{n\to\infty} \Pee{\Big[} \Frag_K(G_n) \in \Fcal {\Big]}
& \geq &
\displaystyle
\lim_{n\to\infty} \Pee{\Big[} \Frag(G_n) \in \Fcal \text{ and }
v(\Frag(G_n)) \leq  K {\Big]} \\
& \geq &
\displaystyle
\lim_{n\to\infty} \Pee{\Big[} \Frag(G_n) \in \Fcal {\Big]} - \lim_{n\to\infty} \Pee{\Big[}
v(\Frag(G_n)) >  K {\Big]} \\
& > &
\Pee( R \in \Fcal ) - \eps.
\end{array} \]
\noindent
This concludes the proof of the lemma.
\end{proof}

Having established that $L_{\FO}$ is a dense subset of
$\{ \Pee( R \in \Fcal ) : \Fcal \subseteq \Ucal\Gcal \}$,
to prove Theorem~\ref{thm:finint} it suffices to show that this
last set is a finite union of closed intervals.
For the remainder of this section the random graph $G_n$ will no longer play any role, and
all mention of probabilities, events etc.~are with respect to
the Boltzmann--Poisson random graph $R$.


Let us order the unlabelled graphs $G_1, G_2, \dots \in \Ucal\Gcal$ in such a way that
the probabilities $p_i := \Pee( R = G_i )$ are non-increasing.
By Corollary~\ref{cor:subsums}, to prove Theorem~\ref{thm:finint} it suffices
to show that $p_i \leq \sum_{j>i} p_j$ for all sufficiently large $i$.
For $k \in \eN$, let us write:
\[ \begin{array}{c}
E_k :=
{\Big\{}\text{$R$ contains no component with $<k$ vertices and exactly one component with $k$ vertices}{\Big\}}, \\
q_k := \Pee( E_k ).
\end{array} \]

\begin{lemma}\label{lem:trivtut}
For every $k \in \eN$, there is a set $A_k \subseteq \eN$ such that
$q_k = \sum_{i \in A_k} p_i$. Moreover, the sets $A_k$ are disjoint.
\end{lemma}

\begin{proof}
Phrased differently, the lemma asks for a
$\Fcal \subseteq \Ucal\Gcal$ such that
we can write $\Pee( E_k ) = \sum_{H \in\Fcal} \Pee( R = H )$.
But this is obvious.
That the sets $A_k$ are disjoint follows immediately from the fact that the
events $E_k$ are disjoint.
\end{proof}

\noindent
For each $k \in \eN$, let $Z_k$ denote the number of components of $R$ of order $k$
and let us write
\[ \mu_k := \sum_{H \in \Ucal\Ccal_k} \frac{\rho^k}{\aut(H)}. \]
\noindent
(Recall that $\Ucal\Ccal_k$ denotes the set of unlabelled, connected graphs
from $\Gcal$ on exactly $k$ vertices.)
Since the sum of independent Poisson random variables is again Poisson-distributed,
it follows from Lemma~\ref{lem:Boltzcomp} that
$Z_1, \dots, Z_k$ are independent Poisson random variables with
means $\Ee Z_i = \mu_i$.
Hence we have:

\begin{equation}\label{eq:mukpk}
q_k = \Pee( \Po( \mu_1) = 0 ) \cdots \Pee( \Po(\mu_{k-1}) = 0 )\Pee( \Po(\mu_k) = 1 ) = \mu_k \upe^{-(\mu_1+\dots+\mu_k)}.
\end{equation}

\begin{lemma}\label{lem:pklim}
We have $\displaystyle \lim_{k\to\infty} q_k = 0$ and ${\displaystyle\lim_{k\to\infty}} \frac{q_{k+1}}{q_k} = 1$.
\end{lemma}

\begin{proof}
For $H \in \Ucal\Ccal_k$ the quantity $k! / \aut(H)$ is exactly the number
of labelled graphs $G \in \Ccal_k$ that are isomorphic to $H$.
It follows that
\[ \mu_k = \sum_{H \in \Ucal\Ccal_k} \frac{\rho^k}{\aut(H)}
= \sum_{G\in\Ccal_k} \frac{\rho^k}{k!} = \frac{|\Ccal_k|}{k!} \cdot \rho^k. \]
\noindent
We thus have that
\[ \sum_{k=1}^\infty \mu_k = C(\rho) \leq G(\rho) < \infty, \]
\noindent
(where $C(z)$ resp.~$G(z)$ denotes the exponential generating function of
$\Ccal$ resp.~$\Gcal$).
In particular we have $\mu_k \to 0$ as $k\to\infty$.
This immediately also gives that $q_k \to 0$ as $k\to\infty$.
Now recall that, according to Theorem~\ref{thm:addissmooth} and the fact that for a decomposable class
$C(.)$ and $G(.)$ have the same radius of convergence, we have
$\frac{(k+1)|\Ccal_k|}{|\Ccal_{k+1}|} \to \rho$ as $k \to\infty$.
We therefore have:
\[
\lim_{k\to\infty} \frac{q_{k+1}}{q_k}
= \lim_{k\to\infty} \frac{\mu_{k+1}}{\mu_k} \cdot \upe^{-\mu_{k+1}}
=
\lim_{k\to\infty} \frac{\rho|\Ccal_{k+1}|}{(k+1)|\Ccal_k|} \cdot \upe^{-\mu_{k+1}}  = 1,
\]
\noindent
as required.
\end{proof}

We are now ready to complete the proof of Theorem~\ref{thm:finint}.

\begin{proofof}{Theorem~\ref{thm:finint}}
As observed previously, it suffices to show that there exists some $i_0 \in \eN$ such that
$p_i \leq \sum_{j>i} p_j$ for all $i \geq i_0$.
By Lemma~\ref{lem:pklim}, there is an index $k_0$ such that $q_{k+1} \geq 0.9 \cdot q_k$
for all $k \geq k_0$.
We now fix an  index $i_0$ with the property that $p_{i_0} < q_{k_0}$.

Let $i \geq i_0$ be arbitrary and let $k \geq k_0$ be the largest index
such that $q_k \geq p_i$.
(Such a $k$ exists since $0 < p_i \leq p_{i_0} < q_{k_0}$ and $q_k\to 0$.)
By choice of $k$ we must have $p_i > q_{k+\ell}$ for all $\ell \geq 1$.
Since $k \geq k_0$ we have that
\[
q_{k+1} + q_{k+2} + \dots
\geq (0.9+(0.9)^2+\dots )p_i =
9p_i > p_i \]
\noindent
Recall that, by Lemma~\ref{lem:trivtut}, there are disjoint sets $A_m \subseteq \eN$ such that
$q_m = \sum_{j\in A_m} p_j$ for all $m \in \eN$.
Because $p_i > q_{k+\ell}$ for all $\ell \geq 1$ and
$(p_n)_n$ is non-increasing, we must have that
$i < j$ for all $j \in A := \displaystyle \bigcup_{m > k} A_m$.
It follows that
\[ p_i
< \sum_{\ell > k} q_{\ell}
= \sum_{j \in A} p_j \leq \sum_{j > i} p_j, \]
\noindent
as required. Since $i\geq i_0$ was arbitrary, this concludes the proof
of Theorem~\ref{thm:finint}.
\end{proofof}

\subsection{Obtaining the closure explicitly for forests}

Here we prove Theorem~\ref{thm:forestclosure} above.
For the class of forests $\Fcal$, it is well known that $\rho=\upe^{-1}$ and $G(\rho)=\upe^{1/2}$,
see for example \cite[Theorem~IV.8]{FlajoletSedgewick} for a framework
in which one can construct such explicit constants,
and \cite[p.~470, Section~2]{bernardinoywelsh} for an explanation of that particular
value.

Our plan for the proof of Theorem~\ref{thm:forestclosure} is of course to
apply Corollary~\ref{cor:subsums}. 
We wish to find an ordering $F_1, F_2, \dots$ of all
unlabelled forests $\Ucal\Fcal$ with the property that the probabilities
$p_i := \Pee( R = F_i )$ are non-increasing
(here, $R$ is the Boltzmann--Poisson random graph corresponding to $\Fcal$), 
and to determine exactly for which values $i$ the inequality
$p_i > \sum_{j>i} p_j$ holds.

To this end, we first `guess' the initial part of the order.
Let the graphs $F_1, \dots, F_9$ be as defined in Figure~\ref{fig:forestorder}.

\begin{figure}[H]
\input{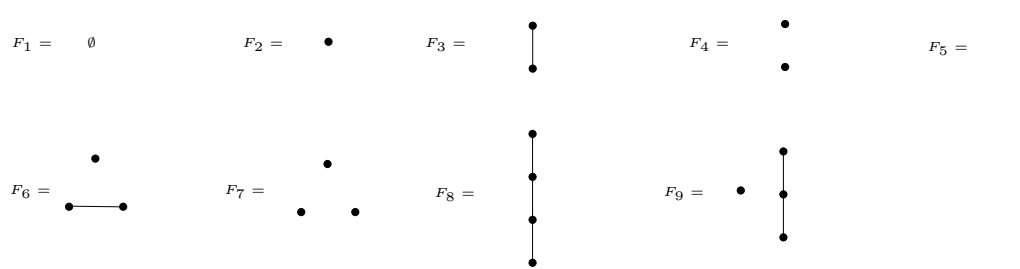}
\caption{The forests $F_1, \dots, F_{9}$.\label{fig:forestorder}}
\end{figure}

\noindent
By substituting $\rho = \upe^{-1}, G(\rho) = \upe^{1/2}$ in~\eqref{eq:Boltzeq}, we find that
the probabilities corresponding to $F_1, \dots, F_9$ are

\begin{equation}\label{eq:forestpkdef}
\begin{array}{lll}
p_1 = \upe^{-1/2}, & p_2 = \upe^{-3/2}, & p_3 = p_4 = \upe^{-5/2}/2, \\
p_5 = p_6 = \upe^{-7/2}/2, & p_7 = \upe^{-7/2}/6, &
p_8 = p_9 = \upe^{-9/2}/2.
\end{array}
\end{equation}

\noindent
It is readily seen that $p_1 \geq \dots \geq p_9$.
Let us remark that, as the reader can easily check, every forest $F$ that is not isomorphic
to one of $F_1,\dots, F_9$ has either five or more vertices or it has exactly four vertices and
$\aut(F) \geq 4$. (In the second case it is either $K_{1,3}$, or four isolated vertices, or
two vertex-disjoint edges, or an edge plus two isolated vertices.)
Hence, if $F$ is not isomorphic to one of $F_1,\dots, F_9$ then
$P( R = F ) \leq \max( \upe^{-9/2}/4, \upe^{-11/2} ) = \upe^{-11/2} < \upe^{-9/2}/2 = p_9$.
This shows that we guessed correctly, and $F_1, \dots, F_9$ are indeed the first nine forests
in our order.

\begin{lemma}\label{lem:autosforest}
$\displaystyle \sum_{F \in \Ucal\Fcal_n} \frac{1}{\aut(F)} > e$ for every $n\geq 6$.
\end{lemma}

\begin{proof}
Let $n \geq 6$. We will explicitly describe enough unlabelled forests on $n$ vertices with
small automorphism groups to make the sum exceed $e$. Figure~\ref{fig:forestauto} shows them
from left to right in the special case $n=6$.

For every $n\geq 6$, the following five forests all have at most two automorphisms: a path on $n$ vertices,
the union of a path on $n-1$ vertices and an isolated vertex, a path on $n-1$ vertices with a leaf
attached to its second vertex, and a path on $n-1$ vertices with a leaf attached to its third vertex,
the union of an isolated vertex and a path on $n-2$ vertices with a leaf attached to its second vertex.
The following forests both have exactly four automorphisms for every $n\geq 6$:
the union of a path on $n-2$ vertices with two isolated vertices,
the union of a path on $n-2$ vertices with a path on two vertices.

\begin{figure}[H]
\begin{center}
\input{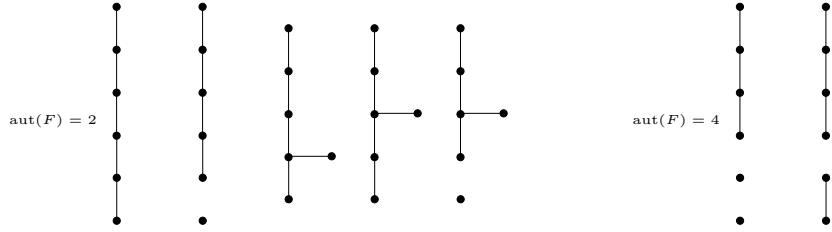}
\end{center}
\caption{Some forests with small automorphism groups.\label{fig:forestauto}}
\end{figure}

\noindent
For every $n\geq 6$, the seven forests just described are pairwise non-isomorphic.
We thus have $\sum_{F \in \Ucal\Fcal_n} \frac{1}{\aut(F)} \geq 5 \cdot (1/2) + 2\cdot (1/4) = 3 > e$.
This proves the lemma.
\end{proof}

\begin{lemma}\label{lem:foresttailfail}
The only indices $k$ for which the inequality $p_k > \sum_{j>k} p_j$ is satisfied
are $k=1,2$.
\end{lemma}

\begin{proof}
Since $\sum_{j>k} p_j = 1 - (p_1+\dots+p_k)$ we have that
$p_k > \sum_{j>k} p_j $ if and only if $p_1+\dots+p_{k-1} + 2p_k > 1$.
The reader can easily check using the expressions given in~\eqref{eq:forestpkdef}
that $k=1,2$ are the only values of $k \leq 9$ for which this inequality holds.

Let $k \geq 10$ be arbitrary, and recall that in this case, as remarked previously,
we have $p_k \leq \upe^{-11/2}$.
Let $n \geq 6$ be the unique integer such that
\[ \upe^{-(n+1/2)} < p_k \leq \upe^{-(n-1/2)}. \]
\noindent
Then $\Pee( R = F ) = \upe^{-(n+1/2)} / \aut(F) < p_k$, for every $F \in \Ucal\Fcal_n$.
In other words, for every $F \in \Ucal\Fcal_n$ there exists a $j > k$ such that $\Pee(R = F ) = p_j$.
In yet other words, every forest on $n$ vertices must come {\em after} position $k$
in our ordering of the unlabelled forests.
Using Lemma~\ref{lem:autosforest} it now follows that:
\[ \sum_{j > k} p_j \geq \sum_{F \in \Ucal\Fcal_n} \Pee( R = F ) = \upe^{-(n+1/2)} \cdot \sum_{F \in \Ucal\Fcal_n} \frac{1}{\aut(F)}
> \upe^{-(n-1/2)} \geq p_k. \]
\noindent
So the inequality $p_k > \sum_{j>k} p_j$ indeed fails for all $k\geq 10$.
This proves the lemma.
\end{proof}

\begin{proofof}{Theorem~\ref{thm:forestclosure}}
By Lemma~\ref{lem:foresttailfail} and Corollary~\ref{cor:subsums}, we see that
\[ \clo(L_\MSO) = \bigcup_{a,b \in \{0,1\}} {\Big[} ap_1+bp_2,
ap_1+bp_2 + (1-p_1-p_2){\Big]}. \]
\noindent
Filling in the values for $p_1, p_2$ from~\eqref{eq:forestpkdef}, we see that we
get exactly the four intervals shown in the statement of the theorem.
\end{proofof}

\subsection{Obtaining the closure explicitly for planar graphs\label{sec:planarlimits}}

Here we prove a more detailed version of
Theorem~\ref{thm:planarclosureprime} above.
The following result is phrased in terms of the exponential generating function for the class of planar graphs, and its radius 
of convergence. Detailed information on these quantities is available in the work of Gim{\'e}nez and
the third author~\cite{GimenezNoy}. In particular both quantities are positive and finite, which can also be seen from the results we included in 
Section~\ref{sec:preliminariesonminorclosedclasses}.

\begin{theorem}\label{thm:planarclosure}
If $\Gcal = \Pcal$ is the class of planar graphs, $\rho$ is the radius of
convergence of its exponential generating function $G$, and if
\begin{equation}\label{eq:leftendpoint}
\lambda_{a,b,c,d,e} := \frac{a+b\rho+\frac{c}{2}\rho^2+(\frac{d}{2}+\frac{e}{6})\rho^3}{G(\rho)}\quad ,
\qquad \ell := 1 - \frac{1+\rho+\rho^2+\frac43\rho^3}{G(\rho)}\qquad ,
\end{equation}
then
\[
\clo( L_\MSO ) =
\bigcup_{a, b \in \{0,1\}, \atop c,d,e \in \{0,1,2\}}
\left[ \lambda_{a,b,c,d,e},\lambda_{a,b,c,d,e} + \ell
\right].
\]
\noindent
In particular, $\clo(L_{\MSO})$ is the union of 108 disjoint intervals each of length
$1 - \frac{1}{G(\rho)}(1+\rho+\rho^2+\frac43\rho^3) \approx 5.39 \cdot 10^{-6}$.
\end{theorem}

The proof closely follows the structure of the proof from the previous section.
Our plan is again to find (the initial part of) an ordering $G_1, G_2, \dots$ of
$\Ucal\Pcal$ such that the sequence of probabilities $p_k := \Pee( R = G_k )$
is non-increasing, and to determine precisely for which indices $k$ the condition
$p_k > \sum_{j>k} p_j$ holds.
Again we start by `guessing' the first few graphs in the ordering.
Let the graphs $G_1, \dots, G_{19}$ be as defined in Figure~\ref{fig:planarorder}.

\begin{figure}[H]
\begin{center}
\input{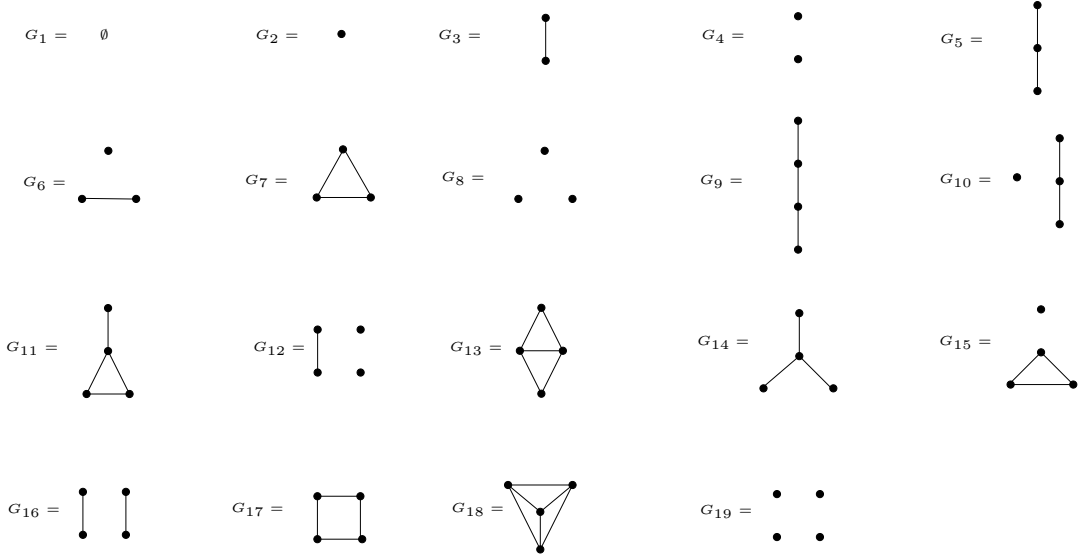}
\end{center}
\caption{The graphs $G_1, \dots, G_{19}$.\label{fig:planarorder}}
\end{figure}

\noindent
(Observe that these are precisely all unlabelled graphs on at most four vertices -- including
the empty graph $G_1$.)
The corresponding probabilities are 

\begin{equation}\label{eq:planarpkdef}
\begin{array}{llll}
p_1 = \frac{1}{G(\rho)},  &
p_2 = \frac{\rho}{G(\rho)},  &
p_3 = p_4 = \frac{\rho^2}{2G(\rho)}, &
p_5 = p_6 = \frac{\rho^3}{2G(\rho)}, \\
p_7 = p_8 = \frac{\rho^3}{6G(\rho)}, &
p_9 = p_{10} = p_{11} = \frac{\rho^4}{2G(\rho)}, &
p_{12} = p_{13} = \frac{\rho^4}{4G(\rho)}, &
p_{14} = p_{15} = \frac{\rho^4}{6G(\rho)}, \\
p_{16} = p_{17} = \frac{\rho^4}{8G(\rho)}, &
p_{18} = p_{19} = \frac{\rho^4}{24G(\rho)}.
\end{array}
\end{equation}

\noindent
To decide for which indices the tail-exceeds-term condition $p_i \leq \sum_{j>i} p_j$ holds
(and to check that $p_1, \dots, p_{19}$ are in non-increasing order and that all graphs on at least five vertices satisfy $p_i \leq p_{19}$),
we need more detailed information on the values of $\rho$ and $G(\rho)$ for planar graphs.
As mentioned above, such information is available form the work of Gim{\'e}nez and
the third author~\cite{GimenezNoy}, who determined both quantities precisely
as the solution of a (non-polynomial) system of equations.
This system in particular enables us to compute
the numbers $\rho$ and $G(\rho)$ to any desired degree of accuracy.
The following approximations suffice for the present purpose:

\begin{lemma}\label{lem:approx}
With $G$ the exponential generating function of labelled planar graphs,
and $\rho$ its radius of convergence,
\[
0.03672841251 \leq \rho \leq 0.03672841266, \quad
0.96325282112 \leq 1/G(\rho) \leq 0.96325282254.
\]
\end{lemma}

Our proof of Theorem~\ref{thm:planarclosure} depends in a delicate way
on the numerical values of $\rho$ and $G(\rho)$.
For example, one can choose approximations of $\rho$ and $1/G(\rho)$ that
agree with Lemma~\ref{lem:approx} in the first five digits,
but differ in the sixth, and that if plugged into the expressions~\eqref{eq:planarpkdef} would
result in a different conclusion for Lemma~\ref{lem:planarplanar} below. This would then suggest a different number of intervals
in the statement Theorem~\ref{thm:planarclosureprime}.
The approximations in Lemma~\ref{lem:approx} can be computed easily using a computer algebra package.
For completeness we provide a proof 
that can be checked by hand, in a supporting document~\cite{support}.

Getting back to the current proof, let us first observe that,
now that we know $\rho < 1/24$, it is indeed true that $p_1 \geq \dots \geq p_{19}$ and that for $i > 19$ we have
$p_i \leq \frac{\rho^5}{G(\rho)} < \frac{\rho^4}{24 G(\rho)} = p_{19}$,
as $G_i$ must have at least five vertices.
Thus $G_1, \dots, G_{19}$ are indeed the first nineteen unlabelled planar graphs,
when the unlabelled planar graphs are ordered by non-increasing value of $\Pee( R = G_i )$.

The following lower bound will be sufficient for our purposes:

\begin{lemma}\label{lem:autos}
For $n \geq 6$ it holds that $\displaystyle \sum_{H \in \Ucal\Pcal_n} \frac{1}{\aut(H)} > 30$.
\end{lemma}
\begin{proof}
For notational convenience let us define
$\varphi(n) :=  \sum_{H \in \Ucal\Pcal_n} \frac{1}{\aut(H)}$. Because of
\[ |\Pcal_n| = \sum_{H \in \Ucal\Pcal_n} \frac{n!}{\aut(H)} = n!\cdot\varphi(n), \]
\noindent
we have $\varphi(n)=\lvert\mathcal{P}_n\rvert/n!$.
Let us observe that
$n\lvert\mathcal{P}_{n-1}\rvert \leq \lvert\mathcal{P}_n\rvert$ for every $n\geq 2$.
This is because given an arbitrary element of $\mathcal{P}_{n-1}$,
already the possibility to add the vertex $n$ and then either join it to exactly
one existing vertex, or leave it isolated, creates $n-1 + 1 = n$ distinct planar labelled
graphs on $[n]$, and all $n\lvert\mathcal{P}_{n-1}\rvert$ elements of $\mathcal{P}_n$ thus
created are distinct.
From this it follows $\varphi$ is monotone non-decreasing in $n$,
as $\varphi(n)/\varphi(n-1) = |\Pcal_n| / (n|\Pcal_{n-1}|) \geq 1$.
It therefore suffices to prove $\varphi(6)>30$, or equivalently, $\lvert\mathcal{P}_6\rvert > 21600$.
Now we use work of Bodirsky, Kang and Gr{\"o}pl:
the number of all labelled planar graphs on six vertices and with $m$
edges is given, for all possible values $0\leq m\leq 12$, in the fifth row of the
table in \cite[Fig.~1]{BodirskyKangGroepl} (the notation $G^{(0)}(n,m)$ is defined
on p.~379). The sum of these numbers is $\lvert\mathcal{P}_6\rvert$, and equals
$32071$, which is strictly greater than $21600$. 
\end{proof}

\begin{lemma}\label{lem:planarplanar}
The only indices $k$ for which the inequality $p_k > \sum_{j>k} p_j$ holds
are $k=1, 2, 4, 6, 8$.
\end{lemma}
\begin{proof}
Recall that $p_k > \sum_{j>k} p_j $ if and only if $p_1+\dots+p_{k-1} + 2p_k > 1$.
As it happens, the estimates on $\rho$ and $1/G(\rho)$ provided by
Lemma~\ref{lem:approx}, together with the expressions~\eqref{eq:planarpkdef}
suffice to determine for which $k \leq 19$ the inequality holds.
We leave the routine arithmetic computations verifying this to the reader.
To see that the inequality holds for $k=1,2,4,6,8$,
it suffices to do explicit calculations with the lower bounds,
and to see that it fails for $k=11,13,15,17,19$, it suffices to do
explicit calculations with the upper bounds provided by Lemma~\ref{lem:approx}.
Observe that $p_k = p_{k+1}$ implies that
$p_k \leq \sum_{j>k} p_j$, so that the inequality automatically fails for
$k=3,5,7,9,10,12,14,16,18$.

To complete the proof of Lemma~\ref{lem:planarplanar}, we are now left with $k \geq 20$.
Let $k \geq 20$ be arbitrary.
Since $G_1,\dots, G_{19}$ are all the unlabelled graphs on at most four vertices,
we must have $v(G_k) \geq 5$, so the formula in Definition~\ref{def:Boltz}
implies $p_k \leq \rho^5 / G(\rho)$. Let $n \geq 6$ be the unique integer such that
\[ \frac{\rho^n}{G(\rho)} < p_k \leq \frac{\rho^{n-1}}{G(\rho)}. \]
\noindent
Then $\Pee( R = H ) = \frac{\rho^n}{\aut(H) G(\rho)} < p_k$ for every $H \in \Ucal\Pcal_n$.
Hence, every graph on $n$ vertices must come {\em after} position $k$ in the ordering.
By Lemma~\ref{lem:autos} and the bound $\rho >  1/30$ from Lemma~\ref{lem:approx},
\[ \sum_{j > k} p_j \geq \sum_{H \in \Ucal\Pcal_n} \Pee( R = H )
= \frac{\rho^{n}}{G(\rho)} \cdot \sum_{H\in\Ucal\Pcal_n} \frac{1}{\aut(H)}
> \frac{\rho^{n-1}}{G(\rho)} \geq p_k, \]
\noindent
completing the proof.
\end{proof}

\begin{proofof}{Theorem~\ref{thm:planarclosure}}
The result follows immediately from Lemma~\ref{lem:planarplanar} via an application of Corollary~\ref{cor:subsums}.
Note that $\sum_{j>8} p_j = 1 - \frac{1}{G(\rho)}(1+\rho+\rho^2+\frac43\rho^3)$ and that $c,d,e$
in the expression given in the theorem take values in $\{0,1,2\}$ because $p_3=p_4$ and $p_5=p_6$ and $p_7=p_8$.

That the $2\cdot 2\cdot 3\cdot 3 \cdot 3 = 108$ intervals thus defined  are all disjoint follows from the
fact that their left endpoints always differ by at least $p_8$ while each interval
has length $\sum_{j>8} p_j < p_8$. (For clarity, we remark here that no additional numerical evaluations are necesary.
All that is needed to deduce the disjointness of the intervals is the statement of Lemma~\ref{lem:planarplanar}. 
See also the remark just after Corollary~\ref{cor:subsums}.)
\end{proofof}

\section{The non-addable case\label{sec:nonaddable}}

\subsection{The $\MSO$-zero-one law for bounded component size \\ (proof of Theorem~\ref{thm:boundedcompsize})}

Here we prove Theorem~\ref{thm:boundedcompsize}.
In this subsection, we fix $t\in\eN$ and $\Gcal$ will be the class of all graphs
whose components have at most $t$ vertices.
We need the following lemma on the number of components of $G_n \in_u \Gcal_n$ isomorphic to a given graph.

\begin{lemma}\label{lem:prut}
Let $H$ be a fixed, connected graph from $\Gcal$, let $K$ be an arbitrary constant and
let $Z_n$ denote the number of components of $G_n$ that are isomorphic to $H$.
Then $Z_n > K$ w.h.p.
\end{lemma}

\begin{proof}
Let us write $r = v(H)$. Using Corollary~\ref{cor:kerstinbounded}, we see that
\[ \Ee Z_n = \frac{{n \choose r} \cdot \frac{r!}{\aut(H)} \cdot |\Gcal_{n-r}| }{ |\Gcal_n| }
\sim \frac{c^r n^{r/t}}{\aut(H)} . \]
\noindent
Similarly, we have
\[ \Ee Z_n(Z_n-1) =
\frac{
{n \choose r} \cdot {n-r \choose r} \cdot \left( \frac{r!}{\aut(H)} \right)^2 \cdot |\Gcal_{n-2r}|
}{
|\Gcal_n|
}
\sim
\left(\frac{c^r n^{r/t}}{\aut(H)}\right)^2.
\]
\noindent
It follows that $\Var(Z_n) = \Ee Z_n^2 - (\Ee Z_n)^2 = o( (\Ee Z_n)^2 )$.
We can thus write
\[ \Pee( Z_n \leq K ) \leq
\Pee( |Z_n - \Ee Z_n | \geq \Ee Z_n / 2 )
\leq 4 \Var Z_n / (\Ee Z_n)^2 = o(1), \]
\noindent
where the first inequality holds for $n$ sufficiently large (as $\Ee Z_n \to \infty$) and we
have used Chebychev's inequality for the second inequality.
Hence $Z_n > K$ w.h.p., as required.
\end{proof}

\begin{proofof}{Theorem~\ref{thm:boundedcompsize}}
Let us fix a $\varphi \in \MSO$ and let $k$ be its quantifier depth.
Let $a = a(k)$ be as provided by Lemma~\ref{lem:lotsofcopies},
let $H_1, \dots, H_m$ be all unlabelled, connected graphs on at most $t$ vertices
and  set $H := aH_1 \cup \dots \cup aH_m$.
(So $H$ is the vertex disjoint union of $a$ copies of $H_i$, for every $i$.)
By Lemma~\ref{lem:prut}, w.h.p., $G_n$ has at least $a$ components isomorphic to $H_i$
for each $i$ (and no other components).
Using Lemmas~\ref{lem:disjointunion} and~\ref{lem:lotsofcopies} it thus follows
that, w.h.p., $G_n \equivMSO_k H$.
So if $H \models \varphi$ then $\lim_{n\to\infty} \Pee( G_n \models \varphi ) = 1$ and
otherwise $\lim_{n\to\infty} \Pee( G_n \models \varphi ) = 0$.
\end{proofof}

\subsection{A $\MSO$-sentence without a limiting probability, for paths \\
(proof of part~\ref{itm:paths1} of Theorem~\ref{thm:paths})\label{sec:nolimpath}}

Note that it is possible to ask,
in a $\MSO$-sentence, for a proper two-colouring of the graph,
such that there are two vertices of degree one with the same colour.
If this sentence is true of a path, the path must have odd order. Otherwise it fails.
Thus we have:

\begin{corollary}\label{cor:onepath}
There exists a $\varphi\in\MSO$ such that
$\displaystyle \Pee( C_n \models\varphi ) =
\left\{ \begin{array}{cl}
1 & \text{ if $n$ is odd, } \\
0 & \text{ if $n$ is even.}
\end{array}\right.$
\end{corollary}

\subsection{The $\FO$-zero-one law fails for caterpillars \\ 
(proof of part~\ref{itm:cat1} of Theorem~\ref{thm:caterpillars})\label{sec:catcounter}}

Let $\varphi$ be the $\FO$ sentence which formalizes
`there are two distinct vertices that have degree five and exactly one neighbour of degree at least two'.
Then clearly a caterpillar satisfies $\varphi$ if and only if
both ends of its spine have degree five.
(Here and elsewhere, we define the spine of a caterpillar
as the path that consists of all vertices of degree at least two.)
The following lemma shows the $\FO$-zero-one law fails for $C_n$, the random caterpillar.

\begin{proposition}\label{prop:caterpillarlimit}
If $\varphi$ is as above, then $\displaystyle \lim_{n\to\infty} \Pee( C_n \models \varphi ) = \left(\frac{\rho^4}{4!(e^\rho-1)}\right)^2$, where $\rho$ is the unique real root of $x \upe^x = 1$.
\end{proposition}


\begin{proof}
An oriented (labelled) caterpillar is a caterpillar, with $\geq 2$ vertices on its spine, on which we choose a
``direction" for the spine.
In other words, an oriented caterpillar is a sequence of at least two stars, such that
the first and the last star have at least two vertices.
Let $a$ and $b$ be the endpoints of the spine. We compute the joint distribution
of the number of leaves attached to $a$ and $b$.
It follows (cf. \cite[Chapter~II]{FlajoletSedgewick}) by basic theory
of exponential generating functions (EGF) that, with number of vertices as size function,
the EGF of stars is $x\upe^x$, of stars with at least two vertices is $xe^x - x$ and of oriented caterpillars is
$$
  C(x)= { (x\upe^x-x)^2\over 1-x\upe^x}.
$$
The numerator encodes the first and last stars,
and the denominator the sequence (possibly empty) of intermediate stars.

Since $C(x)$ is a meromorphic function with a simple pole at $\rho$, we can apply the results from \cite[Section IV.5]{FlajoletSedgewick}). It follows that
the number of oriented caterpillars on $n$ vertices satisfies
\begin{equation}\label{eq:asymptoticsforcaterpillars}
    [x^n] C(x) \sim c \cdot \rho^{-n} n!,
\end{equation}
for some constant $c>0$.

We introduce variables $u$ and $w$ marking, respectively, the number of leaves attached to $a$ and $b$.
The associated EGF is
$$
  C(x,u,w)= { x(\upe^{ux}-1) x(\upe^{wx}-1)\over 1-x\upe^x}.
$$
The probability that $\deg(a)=i+1$ and $\deg(b)=j+1$ is given by
$$
    {[x^n u^i w^j] C(x,u,w) \over [x^n]C(x)}=
    {1\over i!j!}{[x^n] x^{i+j} (\upe^{x}-1)^{-2} C(x) \over [x^n]C(x)}
    \sim {\rho^i \over i! (\upe^\rho-1)}{\rho^j \over j!(\upe^\rho-1)}.
$$
This is because $x^{i+j}(e^x-1)^2$ is analytic, so that
$[x^n] x^{i+j} (\upe^{x}-1)^{-2} C(x) \sim
\rho^{i+j} (\upe^{\rho}-1)^{-2} [x^n] C(x)$.

We see that asymptotically the number of leaves attached to $a$ and $b$ are independent
random variables, each distributed like one plus a Poisson-variable that is conditioned to be positive.
In other words, if~$O_n$ is chosen uniformly at random from all
oriented caterpillars on $n$ vertices then:
\[  \lim_{n\to\infty} \Pee( O_n \models \varphi ) = \left(\frac{\rho^{4}}{4! (e^\rho-1)}\right)^2. \]
Let $E_n$ denote the event that $C_n$, the random (unoriented) caterpillar, is 
not a star (i.e.~the spine has at least two vertices).
Then we have that
\[ \Pee( C_n \models \varphi | E_n ) = \Pee( O_n \models \varphi ),  \]
\noindent
since every unoriented, labelled caterpillar with at least two vertices in the spine
corresponds to exactly two oriented, labelled caterpillars.
Since there are $n$ stars on $n$ vertices, and
the total number of caterpillars is at least $n!/2$ (the number of paths), we have
$\Pee( E_n^c ) = o(1)$.
It follows that
\[ \lim_{n\to\infty} \Pee( C_n \models \varphi ) =
\lim_{n\to\infty} \Pee( C_n \models \varphi | E_n ) \Pee(E_n)
+ \lim_{n\to\infty}\Pee( C_n \models \varphi | E_n^c ) \Pee(E_n^c)
= \left(\frac{\rho^{4}}{4! (e^\rho-1)}\right)^2, \]
\noindent
as claimed.
\end{proof}

\subsection{The $\MSO$-convergence law for forests of paths \\
(proof of part~\ref{itm:paths2} of Theorem~\ref{thm:paths})}

For forests of paths one can easily see that
$C(z) = z + \frac12 \sum_{n=2}^\infty z^n$
(since there is one path on $n=1$ vertex, and there are $n!/2$ paths on $n$ vertices
for all $n\geq 2$).
So the radius of convergence must be $\rho = 1$ and we have $C(\rho) = G(\rho) = \infty$.
In this section and in the following sections we will repeatedly make use of the following
corollary to Lemma~\ref{lem:decompcount}.

\begin{corollary}\label{cor:decompcountinf}
With $\Gcal$, $G_n$ and $\rho$ as in Lemma~\ref{lem:decompcount},
$\Hcal \subseteq \Ucal\Gcal$ any set of (unlabelled) connected graphs from $\Gcal$,
we set $\mu(H) := \rho^{v(H)} / \aut(H)$ for any $H\in\Hcal$
and $\mu(\Hcal) := \sum_{H\in\Hcal} \mu(H)$.
If $\mu(\mathcal{H})=\infty$ then, for any constant $K > 0$,
w.h.p. $G_n$ has at least $K$ components isomorphic to members of $\Hcal$.
\end{corollary}

\begin{proof}
For any finite subset $\Hcal' \subseteq \Hcal$ and any $H\in\Hcal'$
we denote by $N_n(H)$ the number of components of $G_n$ that are
isomorphic to $H$, and by $N_n(\Hcal') = \sum_{H\in\mathcal{H}'} N_n(H)$ the number of
components of $G_n$ isomorphic to some member of $\Hcal'$. A sum of independent
Poisson random variables being Poisson, it follows from Lemma~\ref{lem:decompcount}
that $N_n(\Hcal') \to_{\text{TV}} Z$, where $Z$ is a Poisson random variable with
mean $\Ee Z = \mu(\Hcal') = \sum_{H\in\Hcal'}\mu(H)$. Observe that we can make
$\Ee Z = \mu(\Hcal')$ as large as we wish by taking larger and larger subsets
of $\Hcal$. Using the Chernoff bound (Lemma~\ref{lem:chernoff}) it thus follows that
\[ \limsup_{n\to\infty} \Pee( N_n( \Hcal ) < K )
\leq  \lim_{n\to\infty} \Pee( N_n(\Hcal') < K )
= \Pee( Z < K ) \leq \upe^{ - \Ee Z \cdot H( K/\Ee Z)}, \]
\noindent
(provided $\Ee Z > K$ is sufficiently large), where $H(x) := x \ln x - x + 1$.
In particular, we can make $\Pee( Z < K )$ arbitrarily small
by taking $\Ee Z = \mu(\Hcal')$ sufficiently large.
\end{proof}

The following statement follows immediately from Corollary~\ref{cor:decompcountinf}, since
every path on at least two vertices has exactly two automorphisms and $\rho=1$.

\begin{corollary}\label{cor:AAP}
Let $\Fcal \subseteq \Ucal\Ccal$ be an infinite family of (unlabelled) paths, and let $K > 0$ be an arbitrary
constant.
W.h.p.~the random forest of paths $G_n$ contains at least $K$ components isomorphic
to members of $\Fcal$. \hfill $\blacksquare$
\end{corollary}

Recall that $\equivMSO_k$ is an equivalence relation with finitely many classes.
Let $m$ denote the number of classes that contain at least one path, and let $\Ccal_1, \dots, \Ccal_m$ be a partition
of all (unlabelled) paths according to their $\equivMSO_k$-type.
For each $1 \leq i \leq m$, let us pick an arbitrary representative $H_i \in \Ccal_i$, and let
us denote
\[ \Gamma_k( a_1, \dots, a_m ) := a_1H_1 \cup \dots \cup a_mH_m. \]
\noindent
That is, $\Gamma_k( a_1, \dots, a_m )$ is the vertex-disjoint union of $a_i$ copies of $H_i$, for each $i$.

\begin{proofof}{part~\ref{itm:paths2} of Theorem~\ref{thm:paths}}
Let $\varphi \in \MSO$ be arbitrary and let $k$ be its quantifier depth.

Recall that for $\Hcal \subseteq \Ccal$ we denote by $N_n(\Hcal)$ the
number of components of $G_n$ that are isomorphic to a member of $\Hcal$.
By Lemma~\ref{lem:disjointunion} and the construction of $\Gamma_k(a_1,\dots, a_m)$
we have that
\[ G_n \equivMSO_k \Gamma_k( N_n(\Ccal_1), \dots, N_n(\Ccal_m) ). \]
\noindent
Let us assume (without loss of generality) that the classes
$\Ccal_1, \dots, \Ccal_{m'}$ are finite and $\Ccal_{m'+1}, \dots, \Ccal_m$ are infinite
for some $m' < m$; and let $a = a(k)$ be as provided by Lemma~\ref{lem:lotsofcopies}.
We know from Corollary~\ref{cor:AAP} that $N_n(\Ccal_i) > a$ w.h.p.~for all $i \geq m'$.
Hence, applying Lemma~\ref{lem:disjointunion} and Lemma~\ref{lem:lotsofcopies}, we find that
\[ G_n \equivMSO_k \Gamma_k( N_n(\Ccal_1), \dots, N_n(\Ccal_{m'}), a, \dots, a ) \quad \text{ w.h.p. }
\]
\noindent
Now define
\[ \Acal :=
\{ (a_1, \dots, a_{m'}) \in (\{0\}\cup\eN)^{m'} :
\Gamma_k(a_1, \dots, a_{m'}, a, \dots, a) \models \varphi \}. \]
\noindent
Since a sum of independent Poisson random variables has again a Poisson distribution,
it follows that $(N_n(\Ccal_1), \dots, N_n(\Ccal_{m'}) )\to_{\text{TV}} (Z_1, \dots, Z_{m'})$
where the $Z_i$ are independent Poisson random variables with means
$\Ee Z_i = \sum_{H \in \Ccal_i} 1/\aut(H)$. It follows that
\[
\lim_{n\to\infty} \Pee( G_n \models \varphi )
=
\lim_{n\to\infty} \Pee( (N_n(\Ccal_1), \dots, N_n(\Ccal_{m'})) \in \Acal )
=
\Pee( (Z_1, \dots, Z_{m'}) \in \Acal ). \]
\noindent
This proves the convergence law (since the rightmost expression does not
depend on $n$).
\end{proofof}

\subsection{The $\FO$-convergence law for forests of caterpillars \\
(proof of part~\ref{itm:cat2} of Theorem~\ref{thm:caterpillars})}

For forests of caterpillars we have that $\rho\approx 0.567$,
where $\rho$ is the unique real root of $z\upe^z = 1$
(\cite[Proposition~26]{arXiv:1303.3836v1}),
and $G(\rho) = \infty$ (\cite[Table~1]{arXiv:1303.3836v1}).

Recall that the spine of a caterpillar is the path consisting of all vertices of degree at least two.
For $G$ a forest of caterpillars, let $\Long_\ell(G)$ denote the union of
all components whose spine has $>\ell$ vertices, let
$\Short_{\ell, K}(G)$ denote the union of all components
whose spine has at most $\ell$ vertices and whose degrees
are all at most $K$, and
$\Bush_{\ell, K}(G) := G \setminus (\Long_\ell(G) \cup \Short_{\ell, K}(G) )$
denote the union of all remaining components.
Before starting the proof of the convergence law for forests of caterpillars, we prove some
lemmas on the subgraphs we just defined.

\begin{lemma}\label{lem:bush}
For every $\ell \in \eN$ and $\eps > 0$ there is a $K = K(\ell, \eps)$ such that
$\Pee( \Bush_{\ell, K}(G_n) \neq \emptyset ) \leq \eps$ for $n$ sufficiently large.
\end{lemma}

\begin{proof}
For $s \leq \ell$ and $t$ arbitrary,  let $E_{s,t}$ denote the event
that $G_n$ contains a caterpillar on $s+t$ vertices with $s$ vertices on the spine.
Observe that
\begin{equation}\label{eq:zhuk} \Pee( E_{s,t} ) \leq
\frac{
{n \choose s+t } \cdot {s+t \choose s} \cdot \frac{s!}{2} \cdot s^t \cdot |\Gcal_{n-(s+t)}|
}{
|\Gcal_n|
}.
\end{equation}
\noindent
(We choose $s+t$ vertices and construct a caterpillar of the required kind on them
and a forest of caterpillars on the remaining vertices.
This results in some over-counting, but that is fine for an upper bound.
To construct the caterpillar on the $s+t$ chosen vertices, we first choose $s$ vertices for the spine,
we arrange these $s$ vertices in one of $s!/2$ ways on a path and each of the remaining
$t$ vertices then chooses one of the vertices of the spine to attach to.)

Since the class of forests of caterpillars is smooth (Theorem~\ref{thm:kerstin}),
for every $\eps>0$ there is an $n_0=n_0(\eps)$
such that  $(\rho-\eps) \leq k |\Gcal_{k-1}| / |\Gcal_k| \leq (\rho+\eps)$ for all $ k \geq n_0$.
Let us observe that this also implies that 

\begin{equation}\label{eq:zhuk2}
\frac{|\Gcal_{n-m}|}{|\Gcal_n|} = O\left( \frac{(\rho+\eps)^m}{(n)_m} \right) \quad \quad (\text{for all $n$ and $m \leq n$.})  
\end{equation}

It follows that if $s \leq \ell$ and $n \geq s$ and $t \leq n-s$ are arbitrary then
\[ \begin{array}{rcl}
\Pee( E_{s,t} )
& \leq &
\frac{(n)_{s+t} s^t}{2t!} \cdot \frac{|\Gcal_{n-(s+t)}|}{|\Gcal_n|} \\
& = &
O\left( \frac{s^t (\rho+\eps)^{s+t}}{t!} \right) 
 = O\left( \frac{\left(\ell (\rho+\eps)\right)^{t}}{t!} \right).
\end{array} \]
\noindent
(The first equality holds by rewriting~\eqref{eq:zhuk}, and the second line by filling in~\eqref{eq:zhuk2}.)
We thus have, for every $T$:
\[ 
\Pee\left( \bigcup_{1 \leq s \leq \ell, \atop t \geq T} E_{s,t} \right)
= O\left( 
 \sum_{1\leq s\leq \ell, \atop t \geq T} \frac{\left(\ell (\rho+\eps)\right)^{t}}{t!} \right)
= 
O\left( 
 \sum_{t \geq T} \frac{\left(\ell (\rho+\eps)\right)^{t}}{t!} \right). \]
\noindent
Since $\sum_t \frac{(\ell(\rho+\eps))^t}{t!} < \infty$, we can choose $T$ such that
$\Pee( \bigcup_{1 \leq s \leq \ell, \atop t \geq T} E_{s,t} )  < \eps$ for all $n$.

To conclude the proof of the lemma, we simply set $K(\ell, \eps) = T+2$ and observe that
whenever $\Bush_{\ell, K}(G_n) \neq \emptyset$ then $E_{s,t}$ must hold for some
$s \leq \ell$ and $t \geq T$.
\end{proof}

\begin{lemma}\label{lem:long}
For every $k \in \eN$ there is an $\ell = \ell(k)$ and a forest of caterpillars
$Q_k$ such that $\Long_\ell( G_n ) \equivFO_k Q_k$ w.h.p.
\end{lemma}

The proof of this lemma makes use of an observation that we state as a separate lemma.
Note that the isomorphism type of a caterpillar is described completely be a sequence
of numbers $d_1, \dots, d_\ell$ where $\ell$ is the number of vertices of the spine
and $d_i$ is the number of vertices of degree one attached to the $i$-th vertex of the spine.
We shall call these numbers simply the {\em sequence} of the caterpillar.
Let us say that a caterpillar has {\em begin sequence}
$\overline{d} = (d_1, \dots, d_{\ell})$ (where always $d_1 \geq 1$)
if its sequence either starts with $\overline{d}$ or it ends with
$d_{\ell}, d_{\ell-1}, \dots, d_1$. For every sequence $\overline{d}$ with $d_1 \geq 1$, let
$\Hcal_{\overline{d}} \subseteq \Ucal\Ccal$ denote the set of unlabelled
caterpillars with begin sequence $\overline{d}$.
Throughout this section, $\mu(.)$ will be as defined by Corollary~\ref{cor:decompcountinf} (applied to the class of caterpillars).

\begin{lemma}\label{lem:beginseq}
For any $\ell$ and every sequence $\overline{d} = (d_1,\dots, d_\ell)$ (with $d_1 \geq 1$), we have $\mu(\Hcal_{\overline{d}}) = \infty$.
\end{lemma}

\begin{proof}
Let $\Fcal :=  \Ucal\Ccal \setminus \Hcal_{\overline{d}}$ be the set of those caterpillars that do not have
begin sequence $\overline{d}$.
Recall that $C(\rho) = \sum_{H \in\Ucal\Ccal} \mu(H)$.
Since $\mu(\Fcal) + \mu(\Hcal_{\overline{d}}) = C(\rho) = \infty$, we are done
if $\mu(\Fcal) < \infty$.
Let us thus assume $\mu(\Fcal) = \infty$.

For each $H \in \Fcal$ let us define a graph $H' \in \Hcal_{\overline{d}}$ by adding $\ell+\sum_{i=1}^\ell d_i$ vertices
in the obvious manner.
This clearly defines an injection mapping $\Fcal$ to $\Hcal_{\overline{d}}$.
Also observe that for a caterpillar $H$ with sequence $t_1, \dots, t_k$, the number of automorphisms $\aut(H)$ is either
simply equal to $t_1! \cdots t_k!$ or to
$2 t_1! \cdots t_k!$ (the latter case only occurs if the sequence is symmetric).
Hence, it follows that
\[ \mu(H') \geq \frac{\rho^{\ell+d_1+\dots+d_\ell}}{d_1! \cdots d_\ell!} \cdot \mu(H), \]
\noindent
for every $H \in \Fcal$. It follows that $\mu(\Hcal_{\overline{d}})
\geq \frac{\rho^{\ell+d_1+\dots+d_\ell}}{d_1! \cdots d_\ell!} \cdot \mu(\Fcal) = \infty$, as required.
\end{proof}

Combining this last lemma with Corollary~\ref{cor:decompcountinf}, we immediately get that

\begin{corollary}\label{cor:beginseqprob}
For any $\ell$ and every sequence $\overline{d} = (d_1,\dots, d_\ell)$ (with $d_1 \geq 1$) and every constant $K>0$,
$G_n$ contains at least $K$ components with begin sequence $\overline{d}$, w.h.p.
\end{corollary}

\begin{proofof}{Lemma~\ref{lem:long}}
Let $k \in \eN$ be arbitrary. Recall that, up to logical equivalence,
there are only finitely many
$\FO$-sentences of quantifier depth at most $k$.
Thus, by Gaifman's theorem (Theorem~\ref{thm:gaifman}), there is a finite set
$\Bcal = \{\varphi_1, \dots, \varphi_m \}$ such that every  sentence of
quantifier depth at most $k$ is equivalent to a boolean combination
of sentences in $\Bcal$, and for each $i$ we can write:
\[ \varphi_i =
\exists x_1, \dots, x_{n_i} : \left(\bigwedge_{1\leq a \leq n_i} \psi_i^{B(x_a, \ell_i)}(x_a) \right) \wedge
\left(\bigwedge_{1 \leq a <  b \leq n_i} \dist(x_a, x_b ) > 2\ell_i \right). \]
\noindent
Now let us set $\ell = 1000 \cdot \max_i \ell_i$.
For each $i$, let us fix a caterpillar $H_i$ with a spine of at least $\ell$ vertices that
satisfies $H_i \models  \exists x : \psi_i^{B(x, \ell_i)}(x)$, if such a
caterpillar exists.
Without loss of generality we can assume there exists a $m' \leq m$ such that the sought caterpillar $H_i$ exists
for all $i \leq m'$ and it does not exist for $i > m'$.
Let us now set
\[ Q_k := n_1H_1 \cup \dots \cup n_{m'}H_{m'}. \]
\noindent
That is, $Q_k$ is the vertex disjoint union of $n_i$ copies of $H_i$, for each $i \leq m'$.

For $i > m'$ we have that both $Q_k \models \neg\varphi_i$ and
$\Long_\ell( G_n ) \models \neg\varphi_i$.
(Since there is no caterpillar with
a spine of at least $\ell$ vertices that satisfies $\exists x : \psi_i^{B(x, \ell_i)}(x)$.)
On the other hand, for $i \leq m'$, we have that $Q_k \models \varphi_i$.
Now note that, by the choice of $\ell$ and $H_i$ it is either the case
that {\bf 1)} every caterpillar whose begin sequence is equal to the sequence of $H_i$ satisfies $\exists x : \psi_i^{B(x, \ell_i)}(x)$,
or {\bf 2)} every caterpillar whose begin sequence is equal to the reverse of the sequence of $H_i$ satisfies $\exists x : \psi_i^{B(x, \ell_i)}(x)$.
By Corollary~\ref{cor:beginseqprob}, w.h.p., $G_n$ contains at least $n_i$ components with either begin sequence.
Hence, w.h.p., $\Long_\ell(G_n) \models \varphi_i$.

We have seen that $Q_k \models \varphi_i$ if and only if $\Long_\ell(G_n) \models \varphi_i$ w.h.p. (for all $1\leq i \leq m$).
Since every $\FO$-sentence of quantifier depth at most $k$ is a boolean combination of
$\varphi_1, \dots, \varphi_m$, it follows that
$Q_k \equivFO_k \Long_\ell( G_n )$ w.h.p.
\end{proofof}

\begin{proofof}{part~\ref{itm:cat2} of Theorem~\ref{thm:caterpillars}}
Let us fix a sentence $\varphi \in \FO$, let $k$ be its quantifier depth and let $\ell, Q_k$ be as provided
by Lemma~\ref{lem:long}.
Let $\eps > 0$ be arbitrary, and let $K = K(\ell, \eps)$ be as provided by Lemma~\ref{lem:bush}.
Let $H_1, \dots, H_m \in\Ucal\Ccal$ be all (unlabelled) caterpillars whose spines have
at most $\ell$ vertices and whose sequence has only numbers less than $K$; and let
$N_n(H_i)$ denote the number of components of $G_n$ isomorphic to $H_i$.
By Lemma~\ref{lem:decompcount}, we have that
$(N(H_1), \dots, N(H_m)) \to_{\text{TV}} (Z_1, \dots, Z_m)$, where
the $Z_i$ are independent Poisson random variables with means
$\Ee Z_i = \mu(H_i) = \rho^{v(H_i)}/\aut(H_i)$.
Let us set
\[ \Lambda_k(a_1, \dots, a_m) :=
a_1H_1\cup\dots\cup a_mH_m \cup Q_k. \]
\noindent
(That is, $\Lambda_k(a_1, \dots, a_m)$ is the vertex disjoint union of $Q_k$
with $a_i$ copies of $H_i$ for every $i$.) By Lemma~\ref{lem:long}, we have
that $\Long_\ell( G_n ) \equivFO_k Q_k$ w.h.p.
Since $G_n \setminus \Bush_{\ell, K}(G_n)$ is exactly the union of 
all $\Long_\ell(G_n)$ with all component of $G_n$ that have spine length at most $\ell$
and all degrees at most $K$, it follows using Lemma~\ref{lem:disjointunion} that also
$G_n \setminus \Bush_{\ell, K}(G_n) \equivFO_k \Lambda_k( N(H_1), \dots, N(H_m) )$ w.h.p.
Now let $\Acal := \{ (a_1,\dots, a_m) \in (\{0\}\cup\eN)^m : \Lambda_k(a_1,\dots,a_m) \models \varphi \}$.
We see that
\[ \begin{array}{rcl}
\Pee( G_n \models \varphi )
& \leq &
 \Pee( (N(H_1), \dots, N(H_m)) \in \Acal ) + \Pee( \Bush_{\ell, K}(G_n)\neq\emptyset )
 + \Pee( \Long_\ell(G_n) \not\equivFO_k Q_k ) \\
&  \leq &
 \Pee( (Z_1, \dots, Z_m) \in \Acal ) + \eps + o(1),
 \end{array} \]
\noindent
and, similarly
\[ \begin{array}{rcl}
\Pee( G_n \models \varphi )
& \geq &
 \Pee( (N(H_1), \dots, N(H_m)) \in \Acal ) - \Pee( \Bush_{\ell, K}(G_n)\neq\emptyset ) - \Pee( \Long_\ell(G_n) \not\equivFO_k Q_k )  \\
& \geq &
 \Pee( (Z_1, \dots, Z_m) \in \Acal ) - \eps + o(1).
 \end{array} \]
\noindent
These bounds show that $\limsup_{n\to\infty} \Pee( G_n \models\varphi )$
and $\liminf_{n\to\infty} \Pee( G_n \models \varphi )$ differ by at most
$2\eps$. Sending $\eps\downarrow 0$ then proves that the
limit $\lim_{n\to\infty}\Pee( G_n \models\varphi)$ exists.
\end{proofof}

\subsection{The proof of parts~\ref{itm:paths3} of Theorems~\ref{thm:paths} and~\ref{thm:caterpillars}}

The following general result proves parts~\ref{itm:paths3} of
Theorems~\ref{thm:paths} and~\ref{thm:caterpillars} in a single stroke.
We recall that for decomposable classes, we have the `exponential formula'
$G(z) = \exp(C(z))$, where $C(z)$ is the exponential generating function
for the connected graphs $\Ccal \subseteq \Gcal$.
Moreover, both forests of paths and forests of caterpillars
are smooth by Theorem~\ref{thm:kerstin}.
Both classes are decomposable
and they satisfy $G(\rho)=\infty$
(in view of $C(\rho)=\infty$ \cite[Table~1]{arXiv:1303.3836v1} and $G(z)=\upe^{C(z)}$).
Let us also note that both $C$ and $G$ have the same radius of convergence.

\begin{theorem}\label{thm:divergingegf}
Let $\Gcal$ be a decomposable, smooth, minor-closed class
satisfying $G(\rho) = \infty$ and let $G_n \in_u \Gcal_n$.
Then there is a family of $\FO$ properties $\Phi$ such that their
limiting probabilities exist, and
\[ \clo{\Big(} {\Big\{} \lim_{n\to\infty}\Pee( G_n\models\varphi) :
\varphi \in \Phi{\Big\} } {\Big)} = [0,1].
\]
\end{theorem}

\begin{proofof}{Theorem~\ref{thm:divergingegf}}
By the exponential formula, we also have $C(\rho) = \infty$.
Since $n!/\aut(H)$ is exactly the number of labelled graphs isomorphic to $H$
when $v(H) = n$, we have $\sum_{H \in \Ucal\Ccal_n} \frac{1}{\aut(H)} = |\Ccal_n|/n!$.
With $\mu(H)$ as in Corollary~\ref{cor:decompcountinf},
we get the following alternative expression for $C(\rho)$:
\[ C(\rho) = \sum_n \frac{|\Ccal_n|}{n!} \rho^n = \sum_{H \in \Ucal\Ccal} \mu(H). \]
\noindent
Since $C(\rho) = \infty$, we can define an infinite sequence $\Hcal_1, \Hcal_2, \dots$
of finite, disjoint subsets of $\Ucal\Ccal$ with the property that
$1000 \leq \mu_i := \sum_{H \in \Hcal_i} \mu(H) < 1001$ for each $i$.
To see why the upper bound can be made to hold,
keep in mind the definition of $\mu$,
and also that $\rho\leq 1$
(by a result of Bernardi, Noy and Welsh~\cite{bernardinoywelsh}).

Let us set
\[
E_i := {\Big\{}\text{no component of $G_n$ is isomorphic to an element of $\Hcal_i$}{\Big\}} \quad (i=1,2, \dots),  \]
\noindent
and
\[
F_1 := E_1, \quad \quad F_i := E_1^c \cap \dots \cap E_{i-1}^c \cap E_i \quad (i=2,3,\dots).
\]
\noindent
The events $F_i$ are clearly $\FO$-expressible, and disjoint.

Let us now fix some index $i$.
For $1\leq j \leq i$, let $N_j$ denote the number of components isomorphic
to a graph in $\Hcal_j$.
Since the sum of independent Poisson-distributed random variables is again Poisson-distributed,
it follows from Lemma~\ref{lem:decompcount} that
\[ (N_1,\dots, N_i) \to_{\text{TV}} (Z_1,\dots, Z_i), \]
\noindent
where the $Z_j$ are independent Poisson random variables with $\Ee Z_j = \mu_j$.
We find that
\[ \begin{array}{rcl}
p_i
& := &
\lim_{n\to\infty} \Pee( F_i )  \\
& = &
\Pee( Z_1 > 0 ) \cdots \Pee( Z_{i-1} > 0 ) \cdot \Pee( Z_i = 0 ) \\
& = &
(1-\upe^{-\mu_1}) \cdots (1-\upe^{-\mu_{i-1}}) \cdot \upe^{-\mu_i}.
\end{array} \]
\noindent
Observe that
\[
1 \geq \sum_{j=1}^{i} p_j
= 1 - \lim_{n\to\infty} \Pee( E_1^c \cap \dots \cap E_i^c )
= 1 - (1-\upe^{-\mu_1} )\cdots (1-\upe^{-\mu_i})
\geq 1 - (1-\upe^{-1001})^i. \]
\noindent
By sending $i\to\infty$ we see that $\sum_{i=1}^\infty p_i = 1$.
It follows that, for every $i$: 

\[ \sum_{j > i} p_j = 1 - (p_1+\dots+p_i) = 
 \Pee( Z_1 > 0, \dots, Z_i > 0 ) = (1-\upe^{-\mu_1}) \cdots (1-\upe^{-\mu_{i}})
\]

\noindent
Next, observe that since $\mu_i > 1000$ we have
$\upe^{-\mu_i} < 1 - \upe^{-\mu_i}$. Hence:
\[
p_i = (1-\upe^{-\mu_1}) \cdots (1-\upe^{-\mu_{i-1}}) \cdot \upe^{-\mu_i}
< (1-\upe^{-\mu_1}) \cdots (1-\upe^{-\mu_{i}})
= \sum_{j>i} p_j . \]
\noindent
We can thus apply Lemma~\ref{lem:subsums} to derive that
\[ \left\{ \sum_{i \in A} p_i : A \subseteq \eN \right\} = [0,1]. \]
\noindent
For $A \subseteq \eN$ a {\em finite} set, let us write
$F_A := \bigcup_{i\in A} F_i$.
Then $F_A$ is clearly $\FO$-expressible, and
$\lim_{n\to\infty} \Pee( F_A ) = \sum_{i\in A} p_i$.
Let $\Phi \subseteq \FO$ be all corresponding $\FO$-sentences (i.e., every $\varphi\in\Phi$
defines $F_A$ for some finite $A$).
We have:
\[ {\Big\{} \lim_{n\to\infty} \Pee( G_n \models \varphi ) : \varphi \in \Phi {\Big\}}
=
\left\{ \sum_{i\in A} p_i : A \subseteq \eN \text{ finite }\right\}. \]
\noindent
Finally observe that for every $A \subseteq \eN, \eps > 0$ there is a {\em finite}
$A' \subseteq A$ such that $|\sum_{i\in A} p_i - \sum_{i\in A'} p_i| < \eps$.
In other words, the limiting
probabilities of $\Phi$ are dense in $[0,1]$, as required.
\end{proofof}

\section{Discussion and further work\label{sec:conc}}

Here we mention some additional considerations and open questions that arise naturally from our work.

\medskip 
{\bf $\MSO$ limit laws for surfaces.}
While we were not able to extend our proofs of Theorem~\ref{thm:surface01} and~\ref{thm:surfaceconv} to work for \MSO,
we believe that the results should generalize.

\begin{conjecture}
Let $\Gcal$ be the class of all graphs embeddable on a fixed surface $S$.
Then $C_n \in_u \Ccal_n$ obeys the $\MSO$-zero-one law and $G_n \in_u \Gcal_n$ obeys the $\MSO$-convergence law.
\end{conjecture}

Let us mention that the proof of this conjecture is likely to be much more involved than the proofs of Theorems~\ref{thm:surface01} and~\ref{thm:surfaceconv}.
The proof will probably have to take into account detailed information on the global structure of the largest component,
and it may have to treat different surfaces separately.
In $\MSO$ one can for instance express the property ``$G_n$ has an $H$-minor'' (for any $H$) and by the results in~\cite{chapuyfusygimenezmoharnoy}, the random graph on a surface $S$ will have the ``correct" genus (i.e.~$G_n$ and $C_n$ will not be
embeddable on any ``simpler" surface) with high probability. 
Hence, with high probability, at least one forbidden minor
for embeddability on each simpler surface will have to occur.
In particular, if the $\MSO$-zero-one law/convergence law holds, then the value of the limiting probabilities of some
$\MSO$ sentences will depend on the surface $S$.
Let us also mention that no sensible analogue of Gaifman's locality theorem (Theorem~\ref{thm:gaifman}) for $\MSO$ seems
possible and that, related to this, in the $\MSO$-Ehrenfeucht-Fra\"{\i}ss\'e game the set moves allow Spoiler to
exploit global information. (For instance, if $G$ is $4$-colourable and $H$ is not, then Spoiler
can start by exhibiting a proper 4-colouring of $G$ in four set-moves, and then catch Duplicator by either
exhibiting a monochromatic edge or an uncoloured vertex in $H$. Since the chromatic number is a global characteristic -- there are graphs that
are locally tree-like yet have high chromatic number -- this suggests that if we wish to prove the $\MSO$-zero-one and/or convergence laws
for random graphs on surfaces by considering the Ehrenfeucht-Fra\"{\i}ss\'e game, we may have to come up with a rather
involved strategy for Duplicator.)

 An attractive conjecture of Chapuy et al.~\cite{chapuyfusygimenezmoharnoy} states that for every surface $S$, the random
 graph embeddable on $S$ will have chromatic number equal to four with high probability.
 Since being $k$-colourable is expressible in $\MSO$ for every   fixed $k$, establishing
 our conjecture above can be seen as a step in the direction of the Chapuy et al.~conjecture.
 In~\cite{chapuyfusygimenezmoharnoy} it was already shown that the chromatic number is $\in \{4,5\}$ with high probability.
 Proving the $\MSO$-zero-one law will  imply that the chromatic number of the random graph is either four w.h.p. or five w.h.p. 
 (as opposed to some probability mass being on 4 and some on 5, or oscillating between the two values).


\medskip 
{\bf The limiting probabilities.}
In all the cases where we have established the convergence law in this paper, it turned
out that the closure of the set of limiting probabilities is always a union of finitely many
closed intervals. 
A natural question is whether there are choices of a model of random graphs for which the
convergence law holds, but we end up with a more exotic set. 
The answers happens to be yes.
For instance, if in the binomial random graph model $G(n,p)$ we take $p = \frac{\ln n}{n} + \frac{c}{n}$
then the $\FO$-convergence law holds and we get (see~\cite{logicofrandomgraphs}, Section 3.6.4)
that
\begin{equation}\label{eq:Nymann}
\clo\left(\left\{ \lim_{n\to\infty} \Pee( G(n,p) \models \varphi ) : \varphi \in\FO ) \right\}\right)
=
\left\{ \sum_{i\in A} p_i : A \subseteq \{0\} \cup \eN \right\},
\end{equation}
\noindent
where $p_i := \mu^i e^{-\mu}/ i!$ with $\mu = e^{-c}$.
Observe that $p_i > \sum_{j>i} p_j$ for all sufficiently large $i$.
From this it follows that
the right hand side of~\eqref{eq:Nymann} is homeomorphic to the Cantor set (see~\cite{NymannSaenz}).

Returning to random graphs from minor-closed classes, let us recall that  for every addable, minor-closed
class we have that $G(\rho) \leq e^{1/2}$ by a result of Addario et al.~\cite{addarioberrymcdiarmidreed} and
independently Kang and Panagiotou~\cite{KangPanagiotou}.
In the notation of Section~\ref{sec:limitingprobabilities} we have $p_1 = 1/G(\rho)$, so that
this gives that $p_1 > 1-p_1 = \sum_{j>1} p_j$.
Corollary~\ref{cor:subsums}  allows us to deduce from this that there is at least one ``gap" in the closure of the limiting probabilities.
In the case of forests there are in fact three gaps in total.
We believe that every addable, minor-closed class has at least three gaps, and moreover that the reason is the following.

\begin{conjecture}
If $\Gcal$ is an addable minor-closed class, $G(z)$ its exponential generating function and $\rho$ the radius of convergence, then
$$
    G(\rho) < 1+2\rho.
$$
\end{conjecture}

In the notation of Section~\ref{sec:limitingprobabilities} we have $p_1 = 1/G(\rho)$ and $p_2=\rho/G(\rho)$.
So the conjecture is equivalent to $p_2 > 1-p_1-p_2$, which together with Corollary~\ref{cor:subsums} will
indeed imply that
the closure of the set of limiting probabilities consists of at least four intervals (that is, there are at least three gaps). The difficulty lies on relating the values of $\rho$ and $G(\rho)$.

\medskip 
{\bf $\MSO$-convergence law for smooth and decomposable graph classes.}
We were able to show that for forests of paths, the $\MSO$-convergence law holds. For forests of caterpillars
we were only able to show the $\FO$-convergence law, but we believe the $\MSO$-convergence law should hold too.
What is more, we conjecture that this should be true in general for every minor-closed class that is both smooth and
decomposable.

\begin{conjecture}
For every decomposable, smooth, minor-closed class, the MSO-convergence law holds.
\end{conjecture}

In the proof of the $\MSO$-convergence law for forests of paths, we considered the partition $\Ccal_1, \dots, \Ccal_m$ of
all unlabelled paths $\Ucal\Ccal$ into $\equivMSO_k$-equivalence classes.
We used that whenever $|\Ccal_i|=\infty$, then $\Ncal_n(\Ccal_i)$, the number of components isomorphic to elements of $\Ccal_i$, will
grow without bounds, whereas if $|\Ccal_i| < \infty$, then $\Ncal_n(\Ccal_i)$ will tend to a Poisson distribution.
This essentially relied on the fact that all paths on at least two vertices have precisely two automorphisms.

One would hope that for other smooth and decomposable minor-closed classes some variation of the proof for forests of paths, i.e.~considering
the number of components belonging to each $\equivMSO_k$-equivalence class separately, might lead to a proof of our conjecture.
For paths it was very easy to show a dichotomy between $\Ncal_n(\Ccal_i)$ growing without bounds or $\Ncal_n(\Ccal_i)$
following a Poisson law.
For general smooth, decomposable, minor-closed classes one might still expect a similar dichotomy although the proof,
even for forests of caterpillars, is likely to be more technically involved than in the case of forests of paths.

\medskip 
{\bf Extensions to the $2$-addable and rooted case.}
Possible extensions of our results are to classes of graphs with higher connectivity properties. Call a  minor-closed class  $\Gcal$ \emph{2-addable} if it is addable and it is closed under the operation of gluing two graphs through a common edge. This is equivalent to the fact that the minimal excluded minors  are 3-connected. An MSO-zero-one law should hold for \emph{2-connected} graphs in $\Gcal$, by proving an analogous of Theorem 2.12 for pendant copies of
a 2-connected graph overlapping with the host graph in exactly one specified edge.
Moreover, we believe an MSO-convergence law should hold for \emph{rooted} connected planar graphs, adapting the proof by Woods \cite{woods} for rooted trees (the zero-one law does not hold since, for example, the probability that the root vertex has degree $k$ tends to a constant strictly between 0 and 1). We leave both problems for future research.

\medskip 
{\bf Unlabelled graphs.}
In enumerative combinatorics, unlabelled objects are typically much harder to deal with than labelled ones.
We strongly believe that our results on addable classes will extend to the unlabelled case.

\begin{conjecture}
Let $\Gcal$ be an addable, minor-closed class and let $\Ucal\Gcal$ be the corresponding collection of unlabelled graphs.
The MSO-zero-one law holds for $C_n \in_u \Ucal\Ccal_n$, the random \emph{connected}, unlabelled graph from $\Gcal$.
The MSO-convergence law holds for $G_n \in_u \Ucal\Gcal_n$.
\end{conjecture}

The previous conjecture would follow with the same proofs as in Theorems~\ref{thm:MSO01add} and~\ref{thm:MSOconvadd}, provided the
analogue of Theorem~\ref{thm:pendantcopyadd} on pendant copies holds.
It is believed that this is the case, but so far it has been proved only for so-called subcritical classes \cite{subcritical}.
These include forests, outerplanar and series-parallel graphs, but not the class of planar graphs.

\medskip 
{\bf Analogues of the Rado graph.}
Although we do not formulate any research question in this direction, we cannot resist mentioning
some of our thoughts concerning analogues of the Rado graph, a beautiful mathematical object that is associated with the $\FO$-zero-one law
for the binomial random graph $G(n,1/2)$.
If $\Tscr$ is the set of all $\FO$-sentences $\varphi$ such that $\lim_{n\to\infty}\Pee( G(n,1/2) \models \varphi) = 1$, then
as it happens there is (up to isomorphism) exactly one countable graph that satisfies all sentences in $\Tscr$, namely the Rado graph.
The Rado graph has several remarkable properties, and surprisingly, it is connected to seemingly far-removed branches
of mathematics such as number theory and topology. See~\cite{CameronRado} and the references therein for more
background on the Rado graph.

One might wonder whether, for the cases where we have proved the zero-one law, a similar object might exist.
It follows in fact from general arguments from logic that there will always be at least one countable graph that
satisfies every sentence that has limiting probability one.
What is more, carefully re-examining the proof of Theorem~\ref{thm:MSOMk}, we find that we can construct
such a graph by identifying the roots of $M_1, M_2, \dots$ with $M_k$ as in Theorem~\ref{thm:MSOMk}.
One might hope that, similarly to the case of the Rado graph, our graph is the unique (up to isomorphism) countable graph
that satisfies precisely those $\MSO$ sentences that have limiting probability one.
By some straightforward variations on the construction (for instance by not attaching the consecutive $M_k$-s directly to the root, but
rather by hanging them from the root using paths of varying length) we can however produce an {\em uncountable} family
of non-isomorphic graphs with this property.

\section*{Acknowledgements}

We thank Peter Rossmanith for suggesting the $\MSO$-property without a limiting probability for paths.
We thank Ross Kang, Colin McDiarmid,  Peter Rossmanith and Gilles Schaefer
for helpful discussions related to the paper. 
We also thank the anonymous referees for comments that have improved the paper.

%
%
%

\bibliographystyle{plain}
\bibliography{ReferencesMSO_Resubmit}

\end{document}